\documentclass[a4paper, 10pt, twoside, notitlepage]{amsart}

\usepackage[T1]{fontenc}
\usepackage{lmodern,microtype}
\usepackage{amsmath,amscd}
\usepackage{amssymb}
\usepackage{amsthm}
\usepackage{comment}
\usepackage{graphicx}
\usepackage{epstopdf}
\usepackage{mathrsfs}
\usepackage{xcolor}
\usepackage[ocgcolorlinks, linkcolor=blue]{hyperref}
\usepackage{bookmark}

\DeclareMathOperator{\Ric}{Ric}
\DeclareMathOperator{\Riem}{Riem}

\theoremstyle{plain}
\newtheorem{thm}{Theorem}[section]
\newtheorem*{thm*}{Theorem}
\newtheorem{thmA}{Theorem}

\newtheorem{prop}[thm]{Proposition}

\newtheorem{lem}[thm]{Lemma}
\newtheorem{cor}[thm]{Corollary}

\newtheorem*{conjecture*}{Conjecture}
\theoremstyle{definition}

\newtheorem*{defi*}{Definition}
\theoremstyle{remark}
\newtheorem{rmk}[thm]{Remark}

\numberwithin{equation}{section}
\newcommand {\R} {\mathbb{R}}

\newcommand {\p} {\partial}

\newcommand {\dt} {\partial_t}

\newcommand {\D} {\Delta}

\newcommand{\ol}{\overline}

\newcommand{\norm}[1]{\lVert #1 \rVert}
\newcommand{\divergence}{{\rm div}}
\newcommand{\og}{\overline g}

\newcommand{\wt}{\widetilde}
\newcommand{\abs}[1]{\lvert #1 \rvert}

\newcommand{\s}{\hspace{0.5pt}}
\renewcommand{\a}{\alpha}
\renewcommand{\b}{\beta}

\DeclareMathOperator{\tr}{tr}

\DeclareMathOperator{\Id} {Id}

\newcommand{\kommentar}[1]{}

\title[Recovering a Riemannian manifold from minimal surface areas]{Recovering a Riemannian manifold from minimal surface areas}

\author[C. Carstea]{Catalin Carstea}
\address{Department of Applied Mathematics, National Yang Ming Chiao Tung University,
Hsinchu, Taiwan}
\curraddr{}
\email{catalin.carstea@gmail.com}

\author[M. Lassas]{Matti Lassas}
\address{Department of Mathematics and Statistics, University of Helsinki, Helsinki, Finland}
\curraddr{}
\email{matti.lassas@helsinki.fi}

\author[T. Liimatainen]{Tony Liimatainen}
\address{Department of Mathematics and Statistics, University of Helsinki, Helsinki, Finland}
\curraddr{}
\email{tony.liimatainen@helsinki.fi}

\author[L. Tzou]{Leo Tzou}
\address{School of Mathematics and Statistics, University of Melbourne, Melbourne, Australia}
\curraddr{}
\email{leo.tzou@unimelb.edu.au}

\begin{document}
\raggedbottom

	\begin{abstract}
We study the recovery of a three-dimensional Riemannian manifold from the areas
of minimal surfaces. We show that areas of minimal surfaces determine the metric
on manifolds foliated by minimal surfaces, up to boundary-fixing isometry. This
resolves the conjecture on boundary rigidity for least-area data in an exterior
formulation, even without topological restrictions on the minimal surfaces. We
also prove global rigidity for Einstein metrics on compact three-manifolds with
boundary, assuming \(\pi_1(M,\partial M)=0\) and a non-trapping condition. The
Einstein result includes the Euclidean case.
\end{abstract}
\maketitle

		\medskip


	\tableofcontents

\clearpage
\flushbottom
% Final disclosure wording and funding acknowledgments should be agreed by all authors.

\section*{Disclaimer}
	The main results, arguments, and statements of this paper were developed without the use of AI. AI tools were used only for auxiliary mathematical discussions and for assistance with parts of the writing. The authors take full responsibility for the final mathematical statements and proofs. 
    %The authors will improve the presentation in an updated version following this initial release.

\section{Introduction}\label{sec:introduction}
      In the earlier works \cite{carstea2023inverse,carstea2024calder}, it was  demonstrated that the areas of 2D minimal surfaces determine the geometry of the surfaces themselves up to isometry. In this paper, we address the more challenging problem of recovering the full 3D Riemannian metric from the same area data.

     If $(\wt M,\ol g)$ is a Riemannian manifold and $\Sigma\subset \wt M$ is an embedded smooth minimal surface, we set
    \begin{equation}\label{eq:area_map}
    \mathcal{A}(\Sigma)=\int_\Sigma dV_\Sigma
    \end{equation}
    to be the area of the minimal surface. Here $dV_\Sigma$ is the volume form on $\Sigma$ induced from $\ol g$. We assume that all Riemannian manifolds and minimal surfaces in this paper are oriented, and that all minimal surfaces are compact and connected and have smooth nonempty boundary. Surfaces used for the data are contained in the interior of the enlarged ambient manifold. Orientations and signed unit normals are matched near the common boundary whenever two area maps are compared.

    For minimal surfaces there is no simple parametrization of their boundaries in general. However, if the boundary $\p \Sigma$ of a minimal surface $\Sigma\subset \widetilde M$ is fixed and smooth, we may consider minimal surfaces that are small perturbations of $\Sigma$ as follows. We choose Fermi coordinates $(x,t)$ in $\widetilde{M}$ such that $\Sigma$ corresponds to $\{t=0\}$, and we consider minimal surfaces $\Sigma^f$ that are small perturbations of $\Sigma$ given as graphs in Fermi coordinates,
\[
\Sigma^f = \{ (x, u^f(x)) \,:\, x \in \Sigma \}.
\]
Here $u^f$ is the solution of the minimal surface equation in Fermi coordinates, see \eqref{eq:minimal_surface_general}, with given Dirichlet data $f$. That is
\[
u^f|_{\p \Sigma} = f.
\]
Here we assume that the stability equation corresponding to $\Sigma$ is well-posed, that is, $0$ is not a Dirichlet eigenvalue of the Jacobi operator of $\Sigma$. We then call $\Sigma$ \emph{non-degenerate}. For a surface whose boundary lies in the open known exterior, non-degeneracy can be arranged by a sufficiently small inward shrinking that keeps the point under consideration in its interior, as explained below. It holds automatically for the leaves of the smooth foliations considered here. See Remark \ref{rmk:leaves_nondegenerate}. We also assume that $f \in W$ where
\[
W = \{ f \in C^{\infty}(\p \Sigma) \,:\, \norm{f}_{C^{2,\gamma}(\p \Sigma)} < \delta \}
\]
with $\gamma\in (0,1)$ some constant and $\delta > 0$ sufficiently small (so that the minimal surface equation has a unique small solution). We also write $W=W_\delta$. In the above situation, we simply write
\[
\mathcal{A}(f)=\mathcal{A}(\Sigma^f)
\]
and this is the restriction to smooth boundary values of a smooth map on an open neighbourhood of zero in $C^{2,\gamma}(\p\Sigma)$. If we want to emphasize the fixed $\Sigma$, we also write
\[
\mathcal{A}_\Sigma(f)=\mathcal{A}(\Sigma^f).
\]
For these facts, see \cite{carstea2023inverse} and \cite{busch2025generalized}.

\noindent\textbf{Exterior problem formulation of area data:}
We use the exterior formulation of area data in \cite[Section~3.2]{carstea2024calder} and \cite{busch2025generalized} to state our results.
Let $(M_1,\ol g_1)$ and $(M_2,\ol g_2)$ be compact connected $3$-dimensional Riemannian manifolds with mutual smooth nonempty boundary $\p M$. We assume that $M_j$ is contained in the interior of a slightly larger compact connected manifold $\widetilde{M}_j$, and that these enlargements have an identified exterior
\[
\Omega:=\widetilde M_1\setminus M_1=\widetilde M_2\setminus M_2
\]
which is an attached collar of $\p M$, on which $\ol g_1=\ol g_2$. The metrics are smooth up to the outer boundary. Boundaries of measured surfaces lie in the interior of this exterior collar. Thus the geometry of the exterior region $\Omega$ is known. Our data consist of the areas of minimal surfaces with boundary in $\Omega$ and of their small perturbations, in the following sense. Let $\Sigma_1\subset \wt M_1$ and $\Sigma_2\subset \wt M_2$ be non-degenerate embedded minimal surfaces with $\p\Sigma_1=\p\Sigma_2\subset \Omega$ and
\begin{equation}\label{eq:surfaces_agree_in_Omega}
\Sigma_1\cap \Omega=\Sigma_2\cap\Omega.
\end{equation}
With the normals matched in the exterior, the equality $\ol g_1=\ol g_2$ in $\Omega$ implies that the signed Fermi coordinates of $\Sigma_1$ and $\Sigma_2$ coincide near the common boundary $\p\Sigma:=\p\Sigma_1=\p\Sigma_2$, so that the area maps $\mathcal{A}_{\Sigma_1}$ and $\mathcal{A}_{\Sigma_2}$ are defined on a common set $W_\delta$ of boundary values. We say that \emph{the area data of $\Sigma_1$ and $\Sigma_2$ agree} if
\begin{equation}\label{eq:area_data_agree}
\mathcal{A}_{\Sigma_1}(f)=\mathcal{A}_{\Sigma_2}(f) \quad \text{for all } f\in W_\delta
\end{equation}
for some $\delta>0$. We write
\[
\mathcal{A}_1=\mathcal{A}_2
\]
if for every non-degenerate embedded minimal surface $\Sigma_1\subset \wt M_1$ with $\p \Sigma_1\subset \Omega$ there is a non-degenerate embedded minimal surface $\Sigma_2\subset \wt M_2$ satisfying \eqref{eq:surfaces_agree_in_Omega} whose area data agree with those of $\Sigma_1$, and vice versa.
From the condition $\mathcal{A}_1=\mathcal{A}_2$ we would like to show that there is an isometry $(M_1,\ol g_1)\to (M_2,\ol g_2)$.
These are branch-area data: $\mathcal A_\Sigma$ is the area of the small solution
branch issuing from the specified surface, whether or not another spanning
surface has smaller area. Equality of the minimum area over all spanning
surfaces is not assumed to imply equality of these branch maps.

Our earlier results established that areas determine minimal surfaces themselves:
\begin{thmA}[Recovery of minimal surfaces \cite{carstea2023inverse,carstea2024calder}]\label{thm:minimal_surface_general}
Let $(\Sigma_1, g_1)$ and $(\Sigma_2, g_2)$ be non-degenerate embedded oriented 2D minimal surfaces in oriented 3D Riemannian manifolds $(\wt M_1,\ol g_1)$ and $(\wt M_2,\ol g_2)$ respectively, having a mutual smooth boundary $\p \Sigma\subset \Omega$ and satisfying $\Sigma_1\cap \Omega=\Sigma_2\cap \Omega$. Assume that their area data agree,
\[
 \mathcal{A}_{\Sigma_1}(f)=\mathcal{A}_{\Sigma_2}(f)
\]
for all $f\in W$.
Then there exists an isometry $F: \Sigma_1 \to \Sigma_2$,
\[
F^* g_2 = g_1,
\]
which satisfies $F|_{\partial \Sigma} = \text{Id}$, and
\[
 F^*A_2=A_1,
\]
where $g_j$ is the induced metric and $A_j$ is the second fundamental form of $\Sigma_j\subset (\wt M_j,\ol g_j)$.
\end{thmA}
The passage from area data to the Dirichlet-to-Neumann map is given by
\cite[Lemma~3.2]{carstea2024calder}, which we recall in
Section~\ref{subsec:area_dn}. Theorem~\ref{thm:minimal_surface_general} is
\cite[Theorem~1.2]{carstea2024calder} in this exterior formulation.
Its proof also identifies the Jacobi potential
\[
 q_j=-|A_j|_{g_j}^2-\Ric_{\ol g_j}(N_j,N_j),\qquad F^*q_2=q_1.
\]
The precise lower-order coefficient information used here is recorded in
Section~\ref{subsec:known_coefficients}. In Section~\ref{subsec:k2} we recover
the full second normal derivative of the metric as well.

Our first theorem resolves the following conjecture in the exterior formulation described above:
\begin{conjecture*}[Boundary rigidity for least area data \cite{Tracey}]
Let $(\wt M, \ol g)$ be a Riemannian $3$-ball which admits a foliation by properly embedded,
area-minimizing surfaces.
Suppose that for this foliation and any nearby minimal surface perturbation, we know the areas of the
leaves. Then this information determines $\ol g$ up to boundary fixing isometry.
\end{conjecture*}
The theorem below is even more general and allows non-trivial topology for the manifold and the minimal surfaces. The leaves need not be graphical over an a priori known surface. We assume that the areas of the specified leaves and their nearby minimal perturbations are measured, together with the exterior identification stated below. We do not deduce these branch data from a minimum-area function.

By a \emph{foliation by minimal surfaces} of a region
$\mathcal U\subset\wt M$ we mean a smooth family
$(\Sigma^t)_{t\in I}$, with $I\subset\R$ an open interval, of pairwise disjoint
embedded minimal surfaces with $\p\Sigma^t\subset\Omega$ and
$\mathcal U=\bigcup_{t\in I}\Sigma^t$. We require $\mathcal U$ to be a smooth
embedded three-dimensional manifold with boundary
$B=\bigcup_t\p\Sigma^t$, and the leaf parameter to be a smooth submersion.
Locally the foliation is a product of an interval and a disk, or, near $B$, of an
interval and a half-disk. In the latter coordinates the leaves are $t$-level
sets and their boundaries are given by $x^1=0$. Thus the foliation is smooth
up to the full boundaries of its compact leaves. Neighbourhoods at $B$ are
understood relative to $\mathcal U$. Its interior is open in $\wt M$.

\begin{thm}[Recovery of a 3D Riemannian manifold]\label{thm:main1}
Let  $(\wt M_1,\ol g_1)$ and $(\wt M_2,\ol g_2)$ be as above. Assume that there are sets $\mathcal U_1\subset \wt M_1$ and $\mathcal U_2\subset \wt M_2$ with $M_j\subset \mathcal U_j$ which admit foliations $(\Sigma_1^t)_{t\in I}$ and $(\Sigma_2^t)_{t\in I}$ by minimal surfaces with $\p \Sigma_1^t=\p\Sigma_2^t\subset \Omega$ and $\Sigma_1^t\cap \Omega=\Sigma_2^t\cap \Omega$ for all $t\in I$. Assume that the areas of the foliating minimal surfaces and of their minimal surface perturbations agree,
\begin{eqnarray}\label{eq: area perturbations are the same}
\mathcal{A}_{\Sigma_1^t}(f)=\mathcal{A}_{\Sigma_2^t}(f)
\end{eqnarray}
for all $f\in W_{\delta_t}$ and all $t\in I$. Then there exists an isometry
\[
\ol F: (\mathcal U_1,\ol g_1) \to (\mathcal U_2,\ol g_2),\qquad \ol F^* \ol g_2 = \ol g_1,
\]
which maps $\Sigma_1^t$ onto $\Sigma_2^t$ for every $t\in I$ and which is the identity on a neighbourhood of $\bigcup_{t\in I}\p \Sigma_1^t$ in $\Omega$. If in addition every connected component of $\Sigma_1^t\cap \Omega$ meets $\p \Sigma_1^t$ for every $t\in I$ (for example, if $\Sigma_1^t\cap\Omega$ is a collar neighbourhood of $\p\Sigma_1^t$ in $\Sigma_1^t$), then $\ol F$ is the identity on $\mathcal U_1\cap \Omega$ and $\ol F$ restricts to an isometry $(M_1,\ol g_1)\to (M_2,\ol g_2)$.
\end{thm}
The assumption that the minimal surfaces coincide as sets on the exterior region $\Omega$, i.e., $\Sigma_1^t \cap \Omega = \Sigma_2^t \cap \Omega$, is primarily technical. As explained above, it guarantees that the perturbation data \eqref{eq: area perturbations are the same} of $\Sigma_1^t$ and $\Sigma_2^t$ are parametrized by the same boundary values.

We also prove a local version of the above theorem:
\begin{thm}\label{thm:main_local}
Let $(\Sigma_1, g_1)$ and $(\Sigma_2, g_2)$ be embedded minimal surfaces in $(\wt M_1,\ol g_1)$ and $(\wt M_2,\ol g_2)$ respectively, having a mutual smooth boundary $\p \Sigma\subset \Omega$ and satisfying $\Sigma_1\cap\Omega=\Sigma_2\cap\Omega$. Assume that $\Sigma_1$ and $\Sigma_2$ are contained in foliated regions $\mathcal U_1\subset \wt M_1$ and $\mathcal U_2\subset \wt M_2$, whose interiors contain $\Sigma_j\cap M_j$, and which are foliated by minimal surfaces $(\Sigma_1^t)_{t\in I}$ and $(\Sigma_2^t)_{t\in I}$ with $\Sigma_j=\Sigma_j^{t_0}$ for some $t_0\in I$, $\p \Sigma_1^t=\p\Sigma_2^t\subset \Omega$ and $\Sigma_1^t\cap \Omega=\Sigma_2^t\cap \Omega$ for all $t\in I$.  If condition \eqref{eq: area perturbations are the same} holds for all $f\in W_{\delta_t}$ and all $t\in I$, then there is an isometry $\ol F: (\mathcal U_1,\ol g_1) \to (\mathcal U_2,\ol g_2)$,
\[
\ol F^* \ol g_2 = \ol g_1,
\]
which maps $\Sigma_1^t$ onto $\Sigma_2^t$ for every $t\in I$ and which is the identity on a neighbourhood of $\bigcup_{t\in I}\p \Sigma_1^t$ in $\Omega$.
\end{thm}

\begin{rmk}\label{rmk:leaves_nondegenerate}
The leaves of a foliation by minimal surfaces are automatically non-degenerate, so that the area maps $\mathcal{A}_{\Sigma^t_j}$ in Theorem \ref{thm:main1} are well defined. Indeed, if $N$ is a unit normal field of the foliation and $t$ is the foliation parameter, the lapse function $u=\ol g(\p_t,N)>0$ of the leaf $\Sigma^t$ is a positive solution of the Jacobi equation $Ju:=\Delta_{g_t} u+(\abs{A}^2+\Ric(N,N))u=0$ on $\Sigma^t$, see Proposition \ref{prop: elliptic equation for lapse and shift} below. Let $\varphi>0$ be a first Dirichlet eigenfunction of $-J$ on $\Sigma^t$, that is, $-J\varphi=\lambda_1\varphi$ in $\Sigma^t$ and $\varphi=0$ on $\p\Sigma^t$. Green's formula gives
\[
\lambda_1\int_{\Sigma^t}u\varphi\, dV=\int_{\Sigma^t}\big(\varphi\, J u-u\, J\varphi\big)\,dV=-\int_{\p\Sigma^t} u\,\p_\nu\varphi\, dS>0,
\]
where $\nu$ is the outward unit normal vector of $\p\Sigma^t$ in $\Sigma^t$ and $\p_\nu\varphi<0$ on $\p\Sigma^t$ by the Hopf lemma. Hence $\lambda_1>0$, so $0$ is not a Dirichlet eigenvalue of $J$ and the stability equation of $\Sigma^t$ is well-posed.
\end{rmk}

Finally, we prove a rigidity theorem for Einstein metrics. For this we need a ``non-trapping'' condition.
\begin{defi*}[Non-trapping]
We call the pair $(\wt M,M)$ \emph{non-trapping} for minimal surfaces if every
$x\in M$ lies on a non-degenerate smooth embedded minimal surface
$\Sigma_x\subset\operatorname{Int}\wt M$ with
$\p\Sigma_x\subset\Omega=\wt M\setminus M$. Such a point lies in the interior
of $\Sigma_x$. A point of $M$ for which no such surface exists is called
inaccessible.
\end{defi*}

If \(\widetilde{M}\) and \(M\) are clear from the context, we may simply say that \(\overline{g}\) is non-trapping. The non-degeneracy requirement in the definition is not a serious restriction: if $\Sigma$ is an embedded minimal surface containing $x$ in its interior with $\p\Sigma\subset \wt M\setminus M$, then all but countably many slight shrinkings of $\Sigma$ are non-degenerate, since the Dirichlet eigenvalues of the Jacobi operator are strictly monotone under shrinking of the domain. Such shrinkings still contain $x$ and have their boundary in $\wt M\setminus M$.
The terminology originates from the setting of geodesics, which can be considered as one dimensional minimal surfaces.
A manifold \((\widetilde{M}, \overline{g})\) satisfies the non-trapping condition for geodesics if it contains no trapped geodesics. In particular, every point of $\wt M$ then lies on a geodesic with both of its endpoints in \(\widetilde M \setminus M\).

The non-trapping condition plays a natural role in recovering geometric information about \((\widetilde{M},\overline{g})\) near a point \(x\in \widetilde{M}\) from the areas of minimal surfaces.
A metric perturbation supported away from an existing minimal-surface branch leaves that branch and its areas unchanged. This observation alone does not establish equality of the complete area data, since a metric perturbation may create additional minimal surfaces.

In dimension \(3\), if $M$ is a compact subdomain of an asymptotically flat manifold containing no closed embedded minimal surfaces, the non-trapping condition follows from \cite{chodosh2018asymptotically} (see also \cite{mazet2022minimal}), where a properly embedded minimal plane through every point is constructed.
Subsets of Euclidean space and of hyperbolic space, as well as their small perturbations, also satisfy the non-trapping condition.

Our last result regards rigidity of Einstein metrics. Recall that an Einstein metric $\ol g$ on $\wt M$ is a metric that satisfies
\[
\Ric(\ol g)=\frac{R(\ol g)}{n} \ol g,
\]
where $\Ric(\ol g)$ and $R(\ol g)$ are the Ricci tensor and the scalar curvature respectively and $n=\dim \wt M$. (In dimension $3$, Einstein metrics are exactly the metrics of constant sectional curvature.) We use the condition $\pi_1(M,\p M)=0$ to mean that $\p M$ is connected and every loop in $M$ based at a boundary point can be continuously deformed, with that point fixed, to a loop lying entirely in $\p M$.
It holds if $M$ is simply connected and has connected boundary. The converse
need not hold, a solid torus is an example.

\begin{thm}\label{thm:main2}
Let $\wt M$ be a compact $3$-dimensional manifold with boundary and let $M\subset \wt M$ and $\Omega=\wt M\setminus M$ be as above, with $\pi_1(M,\p M)=0$. Let $\ol g_1$ and $\ol g_2$ be Riemannian metrics on $\wt M$ with $\ol g_1=\ol g_2$ in $\Omega$. Assume that $\ol g_2$ is an Einstein metric and that $(\wt M, \ol g_1)$ is non-trapping.
Then the condition
\[
\mathcal{A}_1=\mathcal{A}_2
\]
implies that there is a diffeomorphism $\ol F: M \to M$ fixing $\p M$ pointwise, with
\[
\ol F^* \ol g_2 = \ol g_1.
\]
The same conclusion holds if $\ol g_2$ is assumed to be Ricci flat instead of Einstein.
\end{thm}
In fact only one half of the condition $\mathcal{A}_1=\mathcal{A}_2$ is used in the proof, namely that for every non-degenerate embedded minimal surface $\Sigma_1$ of $(\wt M,\ol g_1)$ with $\p\Sigma_1\subset\Omega$ there is a non-degenerate minimal surface $\Sigma_2$ of $(\wt M,\ol g_2)$ satisfying \eqref{eq:surfaces_agree_in_Omega} whose area data agree with those of $\Sigma_1$.
The proof is based on showing that also $(\wt M,\ol g_1)$ is Einstein and then using uniqueness of solutions to the Einstein equations under the assumption $\pi_1(M,\p M)=0$. We refer the reader to \cite{anderson2008boundary} for a related boundary value problem for Einstein metrics, and to \cite{marx2025inverse} for a related rigidity result for renormalized areas of minimal surfaces in the asymptotically hyperbolic setting.

As a special case, this theorem gives rigidity for the Euclidean space.
\begin{cor}\label{thm:main3}
Let $\wt M\subset \R^3$, $M$ and $\Omega$ be as in Theorem \ref{thm:main2}, let $\ol g_2=e$ be the Euclidean metric and let $\ol g_1$ be a non-trapping metric on $\wt M$ with $\ol g_1=e$ in $\Omega$.
Then the condition $
\mathcal{A}_1=\mathcal{A}_2$ implies that $(M,\ol g_1)$ is isometric to $(M,e)$.
\end{cor}

\subsection{Motivation}
The generalized boundary rigidity problem asks whether the areas of minimal
surfaces spanned by boundary curves determine the Riemannian metric. In one
dimension this reduces to the classical boundary rigidity problem, whose data
are the lengths of minimizing geodesics. The generalization is natural because
minimal surfaces are a higher-dimensional analogue of geodesics, and their
areas provide a richer set of data.

{\color{blue}
Recent works, which we discuss in Section \ref{subsec:novelty} below, have
shown that replacing geodesics by minimal surfaces leads to genuinely new
phenomena. One such development is the finite-dimensional viewpoint on
generalized boundary rigidity developed in
\cite{Alvarez-Lefeuvre-Lowe-Smith}. There the boundary data are given by areas
of minimal surfaces (or more general surfaces) anchored to a \(3\)-dimensional
family of circles on the boundary of the manifold. As discussed in
\cite[Sections~1.3 and 1.6.2]{Alvarez-Lefeuvre-Lowe-Smith}, this viewpoint fits
into the Gromov--Labourie program of finding higher-dimensional replacements
for the geodesic flow, a program that gained momentum after the Kahn--Markovic
theorem. See also \cite[Section 1.3]{busch2025generalized} for discussion on the
dimensionality of the data.
}

The problem also connects to the AdS/CFT correspondence. This duality, proposed by Maldacena~\cite{maldacena1999large}, establishes an equivalence between conformal field theories (CFTs) and the geometry of Anti-de Sitter (AdS) spacetimes, often called the ``bulk''. A notable proposal within this correspondence, introduced by Ryu and Takayanagi~\cite{ryu2006holographic,ryu2006aspects}, connects the entanglement entropies of a CFT to the areas of minimal surfaces in an AdS spacetime. In this framework, a subregion $A$ on the conformal boundary of the AdS spacetime determines a minimal surface $\Sigma$ in the bulk, anchored at infinity to the boundary $\partial A$ of $A$. The Ryu--Takayanagi conjecture posits that the entanglement entropy of $A$ in the CFT equals $(4G)^{-1}\operatorname{Vol}(\Sigma)$, where $G$ is Newton's constant and $\operatorname{Vol}(\Sigma)$ denotes the volume of $\Sigma$. Entanglement entropy measures the quantum correlations between $A$ and its complement.

In the physics literature, the generalized boundary rigidity problem is called the \emph{bulk reconstruction problem}. Progress has been made in  symmetric settings for example in~\cite{bilson2008extracting,bilson2011extracting,fonda2015shape,jokela2021towards,jokela2019notes, Ji2025}, with formal arguments for reconstruction outlined in~\cite{bao2019towards}. We refer to \cite{jokela2025bulk,carstea2024calder} for further physics references on the matter.

Arguably, progress in bulk reconstruction has largely been limited by the
mathematical tools available in physics. The work~\cite{jokela2025bulk},
building on~\cite{carstea2024calder}, applies methods from inverse problems,
notably higher-order linearization. The present study continues this direction
and develops tools for bulk reconstruction, including the recovery of the full
Riemannian metric.

% Arguably, progress in bulk reconstruction has been constrained by the limited mathematical tools traditionally used in physics. The work~\cite{jokela2025bulk}, building on~\cite{carstea2024calder}, marks a shift by applying modern methods from inverse problems, such as the higher-order linearization method. The present study expands this direction by offering new state-of-the-art tools for bulk reconstruction, including the recovery of the full Riemannian metric and the use of local perturbations of minimal surfaces. These features make the generalized problem both challenging and rich in new applications.

Before turning to the novelties of our results, we also mention the related works
\cite{nurminen2023inverse,nurminen2024inverse,2509.22078} on inverse problems
for minimal and prescribed mean curvature surfaces, and \cite{2606.24682v1} on
the linearized minimal surface problem.

\subsection{Novelty of the results}\label{subsec:novelty}
{\color{blue}
As the field is evolving rapidly, we list below the main novelties compared to
the closest related works \cite{Tracey, carstea2024calder, busch2025generalized,
Alvarez-Lefeuvre-Lowe-Smith} for the reader's convenience. In
\cite{Tracey, carstea2024calder} the data are infinite-dimensional, as in the
present paper, whereas in \cite{busch2025generalized,
Alvarez-Lefeuvre-Lowe-Smith} the data are finite-dimensional, and in fact the
problem in \cite{Alvarez-Lefeuvre-Lowe-Smith} is well-posed in the PDE sense.
The former two works use a Calder\'on-type approach, while the latter two use
integral geometry.
}

 \textbf{Minimal surfaces glued together.} The earlier work \cite{carstea2024calder} recovered only individual minimal surfaces, including their induced metrics and second fundamental forms, but not the ambient metric. Our foliation theorems glue these surfaces together.

{\color{blue}
\textbf{Bulk manifold unknown.} 
The earlier work \cite{Tracey} and the recent
papers \cite{busch2025generalized, Alvarez-Lefeuvre-Lowe-Smith} are formulated
in settings where the bulk manifold is fixed. In particular, in
\cite{Alvarez-Lefeuvre-Lowe-Smith} the manifold is a ball, as in the boundary
rigidity conjecture proposed in \cite{Tracey}, so there is no topology to
recover. Our foliated theorems compare potentially different manifolds and
recover them up to isometry. The problems are complementary:
\cite{Alvarez-Lefeuvre-Lowe-Smith} recovers the metric from a sharp
finite-dimensional family of boundary circles, while our method uses
infinite-dimensional boundary data in exterior formulation.

% The earlier work \cite{Tracey} and the recent
% papers \cite{busch2025generalized, Alvarez-Lefeuvre-Lowe-Smith} are formulated
% in settings where the bulk manifold is fixed. In particular, in
% \cite{Alvarez-Lefeuvre-Lowe-Smith} the manifold is a ball, as in the boundary
% rigidity conjecture proposed in \cite{Tracey}, so there is no topology to recover. 
% Our foliated theorems compare potentially different manifolds and recover them
% up to isometry.  The problems are complementary:
% \cite{Alvarez-Lefeuvre-Lowe-Smith} recovers the metric from a sharp
% finite-dimensional family of boundary circles,
% without compact-support or vanishing-jet assumptions, while our method uses
% infinite-dimensional boundary data in exterior formulation.
}

{\color{blue}
\textbf{No topological assumptions on the surfaces.} 
In both \cite{Tracey} and
\cite{Alvarez-Lefeuvre-Lowe-Smith} the foliating minimal surfaces are discs,
spanned in the latter case by a \(3\)-dimensional family of circles on the
boundary. The finite-dimensional family used there is sufficient for stable
microlocal inversion and is sharp. In our work the foliating minimal
surfaces can be of arbitrary topology and need not be graphical over a fixed
surface. This allows us to treat manifolds with non-trivial topology and to use
such foliations, which is essential for applications to general
\(3\)-manifolds.
}

{\color{blue}
\textbf{Local perturbations are enough.}
 Our method recovers the metric by
varying minimal surfaces locally. In the local theorem we assume the existence
of foliations of neighbourhoods of the reference surfaces whose corresponding
leaves have equal area maps. By contrast, one of the main achievements of
\cite{Alvarez-Lefeuvre-Lowe-Smith} is proving the existence of a global
foliation, with boundary curves lying on the boundary of the manifold rather
than outside it. Their recovery method then relies on the global family of surfaces
they constructed. In a sense our method propagates along the minimal surfaces,
in comparison to the mentioned works and \cite{busch2025generalized}, which
propagate in (co)normal directions. It is possible that the infinite-dimensional
nature of the boundary data is what allows local perturbations to suffice,
though we do not establish this here.
}

\textbf{No curvature assumptions.} The work \cite{Tracey} requires the metric to be either close to Euclidean or to satisfy thinness or straightness conditions. These are restrictive geometric hypotheses that limit the applicability of the result. In contrast, our theorems require no curvature bounds on the metric.

\textbf{No foliation assumptions.} The earlier works \cite{Tracey, busch2025generalized} do not recover a metric without a foliation assumption.  Our Theorem~\ref{thm:main2} establishes global rigidity for Einstein or Ricci-flat metrics on $3$-manifolds under a non-trapping condition, without assuming a foliation condition.

\textbf{Method.} In the foliated results, we show that the recovered induced metric,
second fundamental form and Jacobi potential determine the lapse and shift,
and hence the ambient metric. The additional recovery of the second normal
metric derivative identifies the full ambient Ricci tensor along a measured
surface. This second part is used for Einstein rigidity.

	\subsection*{Acknowledgments}
    M.L. was partially supported by the Advanced
Grant project 101097198 of the European Research Council, Centre of Excellence
of Research Council of Finland (grant 336786) and the FAME flagship of the Research Council of Finland (grant 359186). 
T.L. is partly supported by the Research Council of Finland (Centre of Excellence in Inverse Modelling and Imaging and FAME Flagship, grants 353091 and 359208).

The authors would like to thank Thibault Lefeuvre for helpful discussion.
% Authors: approve final disclosure wording and add funding acknowledgments before submission.

\section{The minimal surface equation and its higher order linearizations}\label{sec:recovery_of_metric_quantities}

We first record the local boundary problem and the quantities recovered in
\cite{carstea2024calder}. We then use the higher-order integral identities to recover the
remaining second normal derivatives of the metric. Throughout the paper,
\(\Delta_g=\divergence_g\nabla_g\) is the nonpositive Laplace--Beltrami operator.
The scalar second fundamental form is
\(A(X,Y)=\ol g(\ol\nabla_XN,Y)\), and \(H=\tr_gA\).
For a surface with boundary, \(\nu\) denotes the unit normal vector of the boundary within the surface. It is taken to be outward unless stated otherwise.

Let \(\Sigma\) be a non-degenerate minimal surface and choose signed Fermi
coordinates relative to a unit normal \(N\). Thus
\begin{equation}\label{eq:fermi_metric}
 \ol g=ds^2+g_s,\qquad g=g_0,\qquad N|_\Sigma=\p_s.
\end{equation}
Greek indices refer to surface coordinates, and Latin indices to ambient
coordinates. In the hypersurface calculations below the same convention is used
in arbitrary dimension. Set
\begin{equation}\label{eq:hkd}
 \begin{split}
 k(s)&=g_s^{-1},\qquad
 h(s)=\tr(g_s^{-1}\p_sg_s),\\
 k^{(l)}&=\p_s^lk(s)|_{s=0},\qquad
 h^{(l)}=\p_s^lh(s)|_{s=0},\qquad l=1,2,3.
 \end{split}
\end{equation}
The tensors \(k^{(l)}\) are contravariant and act on covectors. We write
\(k^{(l)}(du,dv)\) for this pairing. Minimality gives \(h(0)=2H=0\).
Define the scalar density ratio
\begin{equation}\label{eq:rho_definition}
 \rho(x,s)=\left(\frac{\det g_s(x)}{\det g(x)}\right)^{1/2},\qquad
 q=\frac12h^{(1)}.
\end{equation}
The graph of \(u\) has area
\begin{equation}\label{eq:graph_energy}
 \mathcal E(u)=\int_\Sigma
 \rho(x,u)\sqrt{1+k(u)(du,du)}\,dV_g.
\end{equation}
In particular, \(\rho(0)=1\), \(\p_s\rho(0)=0\), and
\(\p_s^2\rho(0)=q\).

\subsection{The local equation and the boundary data}\label{subsec:area_dn} 

The Euler--Lagrange equation of \eqref{eq:graph_energy} is
\begin{equation}\label{eq:minimal_surface_general}
 -\frac{1}{\sqrt{\det g_u}}\,
 \p_\alpha\left(
 \frac{\sqrt{\det g_u}\,g_u^{\alpha\beta}\p_\beta u}
 {\sqrt{1+g_u^{\mu\nu}\p_\mu u\p_\nu u}}
 \right)+\mathfrak f(x,u,du)=0,
\end{equation}
where \(g_u(x)=g(x,u(x))\) and
\[
 \mathfrak f(x,u,du)
 =\frac{(\p_sk)(x,u)(du,du)}{2\sqrt{1+k(u)(du,du)}}
 +\frac12\sqrt{1+k(u)(du,du)}\,h(x,u).
\]
Here \(\p_sk\) is differentiated in the Fermi variable before evaluation at
\(s=u(x)\). The derivative of the graph equation at \(u=0\) is
\(L=-\Delta_g+q\). Its Dirichlet realization is invertible by non-degeneracy.
To apply the implicit function theorem, write \(u=Ef+w\), where \(E\) is a
bounded extension of boundary values and \(w\) has zero boundary trace. The
derivative in \(w\) is \(L\), as a map from zero-trace \(C^{2,\gamma}\) functions
to \(C^{0,\gamma}\) functions. Consequently, for \(0<\gamma<1\) and some
\(\delta>0\), the problem has a unique small solution \(u(f)\) for
\[
 U_\delta=\{f\in C^{2,\gamma}(\p\Sigma):
                   \|f\|_{C^{2,\gamma}}<\delta\}.
\]
The solution map is smooth, and smooth boundary values give smooth solutions.
The branch-area map is \(\mathcal A_\Sigma(f)=\mathcal E(u(f))\).

We use the fixed outward unit normal \(\nu_g\) of \(\p\Sigma\) in
\((\Sigma,g)\) to define
\(\Lambda_{\ol g}(f)=\p_{\nu_g}u(f)|_{\p\Sigma}\).
By \cite[Lemma~3.2]{carstea2024calder}, the area map and the known metric near
\(\p\Sigma\times\{0\}\) determine the DN map. The normal used there depends
on \(g_f\), but conversion to the fixed normal \(\nu_g\) uses only the known
boundary metric and the tangential derivative of \(f\). Thus equality of the
area maps implies equality of \(\Lambda_{\ol g}\) and its derivatives.

\subsection{Previously recovered coefficients}\label{subsec:known_coefficients}

Let \(\Sigma_j\subset(\wt M_j,\ol g_j)\), \(j=1,2\), have matching exterior
portions and equal branch-area maps. Choose their normals to agree in the
exterior. All normal metric derivatives and their surface derivatives then agree
near the common boundary. Theorem~1.2 and the coefficient recovery in
Section~6 of \cite{carstea2024calder} give the boundary-fixing isometry
\(F:\Sigma_1\to\Sigma_2\) of Theorem~\ref{thm:minimal_surface_general}, with
\begin{equation}\label{eq:known_coefficients}
 F^*g_2=g_1,\qquad F^*k_2^{(1)}=k_1^{(1)},\qquad
 F^*h_2^{(1)}=h_1^{(1)},\qquad F^*h_2^{(2)}=h_1^{(2)}.
\end{equation}
In particular,
\begin{equation}\label{eq:q_recovered}
 q_j=\frac12h_j^{(1)},\qquad F^*q_2=q_1.
\end{equation}
These are the external recovery results used below. The boundary-normalization
lemma in that paper is not, by itself, the coefficient-recovery theorem.

The isometry \(F\) is the identity in a sufficiently small common collar of
\(\p\Sigma\). To see this, its differential fixes the boundary tangent and
maps the inward unit normal vector of \(\p\Sigma\) in \((\Sigma_1,g_1)\) to that in \((\Sigma_2,g_2)\). The metrics and
surfaces coincide in the exterior, so these normal vectors coincide. The isometry
therefore fixes the initial data of the inward normal geodesics, and uniqueness
of geodesics gives the assertion. In common coordinates \(F-\Id\), rather than
\(F\), is flat at the boundary. In particular \(DF=I\) there.

Pull the second metric family back by \(F\) and work on
\(\Sigma=\Sigma_1\). Write \(\tilde g_s=F^*g_{2,s}\) and \(g_s=g_{1,s}\).
Then
\[
 \tilde g_0=g_0=g,\qquad \tilde q=q,\qquad
 \tilde k^{(1)}=k^{(1)},\qquad \tilde h^{(2)}=h^{(2)}.
\]
The two full Fermi metric families agree in a collar of \(\p\Sigma\), for
sufficiently small \(s\). The remaining task in this section is to identify
\(k^{(2)}\).

\subsection{Higher-order integral identities}\label{subsec:fourth_area}

We use the higher-order linearization identities of
\cite[Lemmas~3.3--3.5]{carstea2024calder}. We record the formulas needed here
with the convention \(\Delta_g=\divergence_g\nabla_g\) and the fixed outward
normal \(\nu_g\).
Let \(f_1,\ldots,f_4\in C^\infty(\p\Sigma)\), let
\(u_\epsilon=u(\sum_i\epsilon_i f_i)\), and set
\begin{equation}\label{eq:solution_derivatives}
 v^i=\p_{\epsilon_i}u_\epsilon|_0,\qquad
 w^{ij}=\p_{\epsilon_i}\p_{\epsilon_j}u_\epsilon|_0,\qquad
 w^{ijk}=\p_{\epsilon_i}\p_{\epsilon_j}\p_{\epsilon_k}u_\epsilon|_0.
\end{equation}
The first derivatives solve
\begin{equation}\label{linear_eq}
 \begin{cases}
 (-\Delta_g+q)v^i=0&\text{in }\Sigma,\\
 v^i=f_i&\text{on }\p\Sigma,
 \end{cases}
\end{equation}
and every solution derivative of order at least two has zero boundary trace.
Writing \(L=-\Delta_g+q\), the second derivatives satisfy
\begin{equation}\label{eq:second_linearized_equation}
 \begin{split}
 Lw^{ij}={}&\divergence_g\bigl(v^i k^{(1)}dv^j+v^j k^{(1)}dv^i\bigr)\\
 &-k^{(1)}(dv^i,dv^j)-\frac12h^{(2)}v^iv^j,
 \end{split}
\end{equation}
where \((k^{(1)}dv)^\alpha=(k^{(1)})^{\alpha\beta}\p_\beta v\).
The corresponding second-order integral identity is
\begin{align}\label{eq:second_integral_id}
 \int_{\p\Sigma}f_1\p_{\nu_g}w^{23}\,dS_g
 ={}&\int_\Sigma\bigl[v^1k^{(1)}(dv^2,dv^3)
       +v^2k^{(1)}(dv^1,dv^3)\notag\\
 &\hspace{12mm}+v^3k^{(1)}(dv^1,dv^2)
       +\tfrac12h^{(2)}v^1v^2v^3\bigr]dV_g\notag\\
 &-\int_{\p\Sigma}v^1\bigl[
 v^2k^{(1)}(\nu_g^\flat,dv^3)
 +v^3k^{(1)}(\nu_g^\flat,dv^2)\bigr]dS_g.
\end{align}
In particular, the \(h^{(2)}\) term has the positive sign in
\eqref{eq:second_integral_id}.

After the pullback in Section~\ref{subsec:known_coefficients}, the two metrics
have the same \(g,q,k^{(1)},h^{(2)}\). Dirichlet uniqueness in
\eqref{linear_eq} and \eqref{eq:second_linearized_equation} therefore gives
the same \(v^i\) and \(w^{ij}\) for both metrics. Set
\[
 \kappa=k^{(2)}-\tilde k^{(2)},\qquad
 \chi=h^{(3)}-\tilde h^{(3)}.
\]
Subtract the third-order integral identities of
\cite[Lemma~3.5]{carstea2024calder}. The boundary measurements agree, the
boundary coefficient terms cancel because the Fermi metric families agree in a
collar, and all terms involving only the previously recovered coefficients,
\(v^i\) and \(w^{ij}\) cancel. The remaining identity, with the signs and
summation written explicitly, is
\begin{equation}\label{eq:k2 integral identity}
 \int_\Sigma\left[
 \sum_{i<j}v^kv^\ell\kappa(dv^i,dv^j)
       +\frac12\chi v^1v^2v^3v^4\right]dV_g=0.
\end{equation}
Here \(\{k,\ell\}\) is the complement of \(\{i,j\}\) in
\(\{1,2,3,4\}\). The sum has exactly six terms, with no normalization.
The identity is initially real multilinear. Its complex-linear extension permits
complex-valued solutions of \eqref{linear_eq}. No complex-valued minimal
surfaces are required. Since it is an identity between derivatives at zero, it
holds for arbitrary smooth boundary directions. Thus boundary norms of the
special solutions below need not stay in the fixed small data ball: for each
choice of directions, the differentiation parameters are taken sufficiently
small.

\subsection{Determination of \texorpdfstring{$k^{(2)}$}{k(2)}}\label{subsec:k2}

\begin{prop}\label{prop:second_normal_jet}
Under the hypotheses of Theorem~\ref{thm:minimal_surface_general}, its isometry
\(F\) also satisfies \(F^*k_2^{(2)}=k_1^{(2)}\). Consequently all normal
derivatives of the tangential Fermi metric of orders at most two agree after
pullback by \(F\).
\end{prop}
\begin{proof}
We use \eqref{eq:k2 integral identity} on the common surface. Differentiating
the identity \(g_s^{-1}g_s=I\) gives
\[
 gk^{(2)}g=2(\p_sg_s|_0)g^{-1}(\p_sg_s|_0)-\p_s^2g_s|_0.
\]
Together with the definition of \(h^{(1)}\), this implies
\begin{equation}\label{eq:trace_kappa}
 \tr(gk^{(2)})
   =\tr\bigl(g^{-1}(\p_sg_s|_0)g^{-1}(\p_sg_s|_0)\bigr)-h^{(1)}
   =4|A|_g^2-h^{(1)},\qquad \tr(g\kappa)=0.
\end{equation}
The last equality follows by subtracting the trace identity for the two
metric families. Their induced metrics, second fundamental forms and
\(h^{(1)}\) already agree by \eqref{eq:known_coefficients}.
The tensor \(\kappa\) is real and symmetric and vanishes in a collar of the
boundary, by Section~\ref{subsec:known_coefficients}.

Choose once and for all a real nontrivial solution \(v^1\) of
\eqref{linear_eq}, by prescribing a nonzero real boundary value. For
\(v^2,v^3,v^4\), use the three complex geometric optics solutions from
\cite[Section~4.2 and Proposition~5.1]{carstea2024calder}:
\begin{equation}\label{eq:sols_for_second_proof}
 \begin{split}
 v^2&=e^{(\Psi+\Phi)/h}(a+r_h),\\
 v^3&=e^{(-\Psi+\Phi)/h}(a+r'_h),\\
 v^4&=e^{-2\overline\Phi/h}(\overline a+\widehat r_h).
 \end{split}
\end{equation}
Here \(h\downarrow0\). For a point \(p\) in the dense set of admissible
critical points of that construction, \(\Phi\) is holomorphic Morse with
\(\Phi(p)=0\) and \(d\Phi(p)=0\). The holomorphic function \(\Psi\) has
nonvanishing differential, and the choices ensure that
\(\Theta_2=\Psi+\Phi\) and \(\Theta_3=-\Psi+\Phi\) have no critical points.
The amplitude satisfies \(a=1\) to sufficiently high order at \(p\), and
vanishes to sufficiently high order at every other critical point of \(\Phi\).
The orders can be chosen as large as required by the cited construction.
The remainder estimates there imply, in particular, uniform \(L^3\) bounds for
the three amplitude factors in \eqref{eq:sols_for_second_proof}.

Put \(P_h=v^2v^3v^4\). Its total exponential is
\(e^{2(\Phi-\overline\Phi)/h}\), of modulus one. H\"older's inequality therefore
gives \(\|P_h\|_{L^1}\leq C\). The three terms in
\eqref{eq:k2 integral identity} that contain \(dv^1\) combine into
\(\kappa(dP_h,dv^1)\). Integration by parts gives
\begin{equation}\label{eq:cgo_lower_terms}
 \int_\Sigma\kappa(dP_h,dv^1)\,dV_g
 =-\int_\Sigma P_h\nabla_\alpha
                  (\kappa^{\alpha\beta}\p_\beta v^1)\,dV_g=O(1).
\end{equation}
There is no boundary term because \(\kappa\) vanishes near \(\p\Sigma\).
Also \(\int_\Sigma\chi v^1P_h\,dV_g=O(1)\). Thus
\begin{equation}\label{eq:kappa_O1}
 T_h:=\int_\Sigma v^1\bigl[
 v^4\kappa(dv^2,dv^3)+v^3\kappa(dv^2,dv^4)
                  +v^2\kappa(dv^3,dv^4)\bigr]dV_g=O(1).
\end{equation}

We now use the precise tensor asymptotic, not merely the norm bounds for
individual remainders. Proposition~5.1, equation~(104), of
\cite{carstea2024calder} applies to the real symmetric trace-free tensor
\(K=v^1\kappa\), since it is flat at the boundary. It gives
\begin{equation}\label{eq:cgo_external_asymptotic}
 T_h=\frac1{h^2}\int_\Sigma e^{4i\operatorname{Im}\Phi/h}
     a^2\overline a\,K(d\Theta_2,d\Theta_3)\,dV_g+o(h^{-1}).
\end{equation}
The coefficient in the integral on the right is smooth and independent of
\(h\). Its boundary terms vanish and the amplitude suppresses the other
critical points. Choose the local holomorphic coordinate
\(z=\Psi-\Psi(p)\). Then
\(d\Theta_2(p)=dz\) and \(d\Theta_3(p)=-dz\).
Stationary phase at the nondegenerate critical point \(p\) gives
\begin{equation}\label{eq:cgo_localized_asymptotic}
 T_h=\frac{c_p}{h}K_p(dz,dz)+o(h^{-1}),\qquad c_p\ne0.
\end{equation}
This is exactly the localization used in
\cite[Section~6.2]{carstea2024calder}. The weaker remainder \(o(h^{-1})\)
is sufficient. Multiplication by \(h\), followed by \eqref{eq:kappa_O1},
yields \(K_p(dz,dz)=0\).

In the isothermal coordinate \(z=x+iy\), reality, symmetry and trace-freeness
give
\[
 K^{xx}+K^{yy}=0,\qquad
 K(dz,dz)=K^{xx}-K^{yy}+2iK^{xy}.
\]
Hence \(K(p)=0\). The construction applies on a dense subset of the interior,
as explained in \cite[Section~6.2]{carstea2024calder}, so continuity gives
\(v^1\kappa=K=0\) everywhere. Unique continuation principle for
\(-\Delta_g+q\) implies that the nonzero set of the fixed nontrivial solution
\(v^1\) is dense. Thus \(\kappa=0\) on that set and, by continuity, on all
of \(\Sigma\). The inverse-metric identity at the start of the proof then
identifies \(\p_s^2g_s|_0\), as claimed.
\end{proof}

We do not need to recover \(\chi\) for the geometric results of this paper.
The new information is \(k^{(2)}\), in addition to
\(g\), \(k^{(1)}\), \(q\), and \(h^{(2)}\) in
\eqref{eq:known_coefficients}. The foliated rigidity proof uses only
\(g\), \(A\), and \(q\). The additional normal metric derivative is used
in the recovery of tangential ambient Ricci curvature and the Einstein theorem.

\section{Recovery of curvature terms}\label{sec:recovery_of_curvature_components}

We use the curvature convention
\[
 \Riem(X,Y)Z=\ol\nabla_X\ol\nabla_YZ-\ol\nabla_Y\ol\nabla_XZ
                         -\ol\nabla_{[X,Y]}Z.
\]
Let $\Sigma_j\subset \wt M_j$ for $j=1,2$ be embedded minimal surfaces in $(\wt M_j,\ol g_j)$, and $N_j$ their unit normals.  We define the ``tangential'' $1$-forms $\Theta_j\in \mathcal{T}^{\s 1}_0(\Sigma_j)$ and $2$-tensors $\Xi_j\in \mathcal{T}^{\s 2}_0(\Sigma_j)$ by
\[
 \Theta_j(X_j):= \Ric (\ol g_j)(N_j,X_j), \quad X_j\in T\Sigma_j,
\]
and
 \[
  \Xi_j(X_j,Y_j):=\Riem(\ol g_j)(X_j,N_j,N_j,Y_j\big), \text{ for } X_j,Y_j\in T\Sigma_j,
 \]
 for $j=1,2$.
Here $\Ric(\ol g_j)$ and $\Riem(\ol g_j)$ are the Ricci and Riemannian curvature tensors of  $(\wt M_j,\ol g_j)$. We  consider the normal component of the Ricci curvature $\Ric (\ol g_j)(N_j,N_j)$ as a function on $\Sigma_j$.

We also denote by $R(\ol g_j)$ the scalar curvature of $(\wt M_j,\ol g_j)$.

In this section we prove the following:
\begin{prop}\label{prop:slice_curvature}
 Let $(\Sigma_1,g_1)\subset \wt M_1$ and $(\Sigma_2,g_2)\subset \wt M_2$ be non-degenerate embedded minimal surfaces with mutual boundary $\p \Sigma\subset\Omega$ and $\Sigma_1\cap\Omega=\Sigma_2\cap\Omega$, and assume that their area data agree, that is, $\mathcal A_{\Sigma_1} (f) = \mathcal A_{\Sigma_2}(f)$ for all $f\in W$.
  Then the isometry $F:(\Sigma_1,g_1)\to (\Sigma_2,g_2)$ of Theorem \ref{thm:minimal_surface_general} satisfies
 \begin{align}
  \big(F^*\Ric(\ol g_2)\big)(X,Y)&=\Ric (\ol g_1)(X,Y), \label{eq:Ricci_to_tangential} \\
  F^*\big(\Ric (\ol g_2)(N_2,N_2)\big)&=\Ric (\ol g_1)(N_1,N_1), \label{eq:Ricci_to_normal} \\
  (F^*\Theta_2)(X)&=\Theta_1(X), \label{eq:mixed_Ricci} \\
  (F^*\Xi_2)(X,Y)&=\Xi_1(X,Y), \label{eq:Riemann_to_normal}
  \\
  F^*(R(\ol g_2))&=R(\ol g_1)\label{eq:scalar_curvature}
 \end{align}
 for $X,Y\in T\Sigma_1$.
\end{prop}

The proof is by careful computations in Fermi coordinates and using the quantities recovered by Theorem \ref{thm:minimal_surface_general} and in Section \ref{sec:recovery_of_metric_quantities}.

\subsection{Computations in Fermi coordinates}
In this subsection $\Sigma$ is an embedded hypersurface in a Riemannian manifold $(\wt M, \ol g )$ of dimension $n$.
As agreed in Section \ref{sec:recovery_of_metric_quantities}, Greek letters $\a,\b,\ldots $ denote the indices $1,\ldots, n-1$ of coordinates on $\Sigma$, and Latin letters $a,b,j,l,\ldots$ the indices $1,\ldots,n$ of coordinates on $\wt M$.  In Fermi coordinates $(x,x_n)$ relative to $\Sigma$ the metric reads
\[
 \og=dx_n^2+g_{\alpha\beta}(x,x_n)dx^\alpha dx^\beta.
\]
Thus
\[
 \og_{nn}=\og^{nn}=1 \text{ and } \og_{n\a}=\og^{n\a}=0,
\]
 for  $\a=1,\ldots, n-1$.
Also
\[
 \og_{\a\b}=g_{\a\b} \text{ and } \og^{\a\b}=g^{\a\b},
\]
and
\[
 \p_n\og_{n\a}=0  \text{ and } \p_n\og_{nn}=0.
\]

The scalar second fundamental form $A_{\alpha\beta}=\langle \ol\nabla_{\partial_{\alpha}} N, \partial_{\beta}\rangle$ and mean curvature $ H=g^{\alpha\beta}A(\partial_{\alpha},\partial_{\beta})$ are in Fermi-coordinates given by
\begin{align}\label{eq:2nd_fund}
\begin{split}
A(\partial_{\alpha}, \partial_{\beta}) &=\langle \nabla_{\partial_{\alpha}} \partial_n, \partial_{\beta}\rangle= \Gamma_{\alpha n}^l \ol g_{l\beta} = \frac{1}{2} \ol g^{lm}(\partial_{\alpha} \ol g_{nm} + \partial_n \ol g_{\alpha m} - \partial_m \ol g_{\alpha n}) \ol g_{l\beta} \\
 &= \frac{1}{2} \partial_n g_{\alpha \beta},
 \end{split}
\end{align}
and
\begin{equation}\label{eq:mean_curvature}
H = \frac{1}{2} g^{\alpha \beta} \partial_n g_{\alpha \beta}.
\end{equation}
Note also that since $N=\p_n$
\begin{align*}
(\mathcal{L}_N \ol g)_{\a\b} &= N^k \partial_k \ol g_{\a\b} +\partial_\a N^k \ol g_{k\b} + \partial_\b N^k \ol g_{\a k} = \partial_n g_{\a\b}=2A(\p_\a,\p_\b),
\end{align*}
since $\partial_\a N^k=\partial_\b N^k=0$.

We have for $a=1,\ldots,n$
\[
 \Gamma_{nn}^a=\frac 12 \ol g^{ab}(\p_n \ol g_{bn} + \p_n \ol g_{bn} - \p_b \ol g_{nn})=0,
\]
since $g_{nn}=1$ and $\og_{n\a}=0$.
We also have
\[
 \Gamma_{an}^a=\p_n \log \abs{\ol g}^{1/2}
\]
by the usual contraction formula for the Christoffel symbols.  Thus
\begin{equation}\label{eq:gamma_vert_contract}
 \Gamma_{an}^a=\p_n \log \abs{\ol g}^{1/2}=\p_n \log \abs{g}^{1/2}=\frac 12 \tr(g^{-1}\p_ng).
\end{equation}
For  $a,b=1,\ldots,n$ we  calculate that
\begin{equation}\label{eq:Gammaban}
 \Gamma_{an}^b=\frac 12 \ol g^{bc}(\p_a \ol g_{nc} + \p_n \ol g_{ac} - \p_c \ol g_{an})=\frac 12 \ol g^{bc}\p_n \ol g_{ac}.
\end{equation}
If $(a,b)=(\alpha,\beta)$, this is
\begin{equation}\label{eq:gamma_tang}
 \Gamma_{\alpha n}^\beta=\frac 12 \ol g^{\beta c}\p_n \ol g_{\alpha c}=\frac 12 g^{\beta\gamma}\p_n g_{\alpha\gamma},
\end{equation}
since $g_{nn}=1$ and $g_{n\alpha}=0$. For the same reason
\begin{equation}\label{eq:gamma_n}
 \Gamma_{\alpha n}^n=\Gamma_{n n}^a=0.
\end{equation}
Using the above identities
\[
\Gamma_{nb}^a\Gamma_{an}^b=\frac14 g^{\alpha\gamma}g^{\beta\delta}(\p_n g_{\beta\gamma})(\p_n g_{\alpha\delta})=|A|_g^2.
\]
 Finally we have
\begin{equation}\label{eq:Gamma_squared}
 \Gamma^n_{\alpha\beta}=\frac{1}{2} \ol g^{nm}(\partial_{\alpha} \ol g_{m\b} + \partial_\b \ol g_{m\a} - \partial_m \ol g_{\alpha \b})=-\frac 12 \p_n g_{\a\b}=-A(\p_\a,\p_\b).
\end{equation}

Let us then derive formulas for curvature tensors. By  \eqref{eq:gamma_vert_contract} and \eqref{eq:gamma_n} we obtain
\begin{multline*}
 \Ric_{\ol g}(N,N)=\Ric_{\ol g}(\p_n,\p_n)=\p_a\Gamma^a_{nn}-\p_n\Gamma^a_{an}+\Gamma^a_{nn}\Gamma_{ba}^b-\Gamma_{an}^b\Gamma_{nb}^a \\
 =-\frac 12 \p_n \tr(g^{-1}\p_ng)- |A|_g^2 \\
 = \frac 12 \tr((g^{-1}\p_ng) g^{-1} \p_ng)-\frac 12  \tr(g^{-1}\p_n^2g)-\abs{A}^2=-\frac 12  \tr(g^{-1}\p_n^2g)+\abs{A}^2,
\end{multline*}
where $A_{\alpha\beta}=\frac 12 \p_ng_{\alpha\beta}$
is the second fundamental form as before.

Recall from \eqref{eq:q_recovered} (see also \cite[Section 1.1]{carstea2023inverse}) that the potential $q$ of the first linearization (stability equation) $(-\Delta_g+q)v=0$ in Fermi coordinates is given by
\[
q=\frac 12 \p_n\tr(g^{-1}\p_ng).
\]
Thus
\begin{align}\label{eq:q_in_terms_of_curvature}
 q&=-\frac12\tr(g^{-1}(\p_ng)g^{-1}(\p_ng))
       +\frac12\tr(g^{-1}\p_n^2g)\notag\\
  &=\frac12\tr(g^{-1}\p_n^2g)-2|A|_g^2\notag\\
  &=-\Ric_{\ol g}(N,N)-|A|_g^2.
\end{align}
Thus the linearized equation \eqref{linear_eq} is the Jacobi equation
\[
\Delta_gv+\big(\abs{A}^2+\Ric_{\ol g}(N,N)\big)v=0
\]
of $\Sigma$, cf.~\cite{colding2011course}.

Next we calculate $\ol g(\Riem_{\ol g}(X,N)N,Y)$ and $\Ric_{\ol g}(X,N)$.
We set $X=\p_\alpha$ and $Y=\p_\beta$. By using $\Gamma^a_{nn}=0$, we have
\begin{multline}
 \ol g(\Riem_{\ol g}(\p_\alpha,N)N,\p_\beta)=R_{\alpha nn\beta}=\ol g_{\beta a}R_{\alpha n n}^{\phantom{\alpha n n}a}\\
 =\ol g_{\beta a}\left[\p_\alpha\Gamma^a_{n n}  -\p_n\Gamma^a_{\alpha n}+(\Gamma^b_{nn}\Gamma^a_{\alpha b}-\Gamma^b_{\alpha n}\Gamma^a_{n b})   \right]=-\ol g_{\beta a}\left[\p_n\Gamma^a_{\alpha n}  +\Gamma^b_{\alpha n}\Gamma^a_{n b}    \right].
 \end{multline}

 By using \eqref{eq:gamma_tang} and \eqref{eq:gamma_n}, this is
\begin{multline}\label{eq:curvature_N_components}
  \ol g(\Riem_{\ol g}(\p_\alpha,N)N,\p_\beta)
  =-\frac 12 \ol g_{\beta a}\p_n (\ol g^{ab}\p_n \ol g_{\alpha b})-\frac 14 \ol g_{\beta a} \ol g^{bc}\p_n \ol g_{\alpha c}\ol g^{ad}\p_n \ol g_{ bd}\\
  =-\frac 12  g_{\beta \kappa}\p_n (g^{\kappa \gamma }\p_n g_{\alpha \gamma})-\frac 14 g^{\kappa \gamma }\p_n g_{\alpha \gamma }\p_n g_{ \kappa \b } \\
  =\frac 12  g_{\beta \kappa}g^{\kappa\sigma}(\p_ng_{\sigma\eta})g^{\eta\gamma}\p_n g_{\alpha \gamma}-\frac 12 \p_n^2g_{\a\b}-\frac 14 g^{\kappa \gamma }\p_n g_{\alpha \gamma }\p_n g_{ \kappa \b } \\
  =\frac 14  (\p_ng_{\beta\eta})g^{\eta\gamma}\p_n g_{\alpha \gamma}-\frac 12 \p_n^2g_{\a\b} = (A^2)_{\a\b}-\frac 12 \p_n^2 g_{\a\b},
  \end{multline}
  where $A^2=Ag^{-1}A$. One can also compute that $\p_n^2 g_{\a\b}=\mathcal{L}_N^2\ol g_{\a\b}$.
We may write \eqref{eq:curvature_N_components} using the quantity $k^{(2)}$ in \eqref{eq:hkd}. Indeed, we have
\begin{multline}\label{eq:k2_formula}
k^{(2)}=\p_n^2(g^{-1})=\p_n(-g^{-1}(\p_n g)g^{-1})\\
=2g^{-1}(\p_ng)g^{-1}(\p_ng)g^{-1}-g^{-1}(\p_n^2g)g^{-1}
\end{multline}
and thus
\[
 gk^{(2)}g=2(\p_ng)g^{-1}(\p_ng)-\p_n^2g=8A^2-\p_n^2g.
\]
Consequently
\begin{equation}\label{eq:curvature_N_components_witk_k2}
 \ol g(\Riem_{\ol g}(\p_\alpha,N)N,\p_\beta)=\frac 12 (gk^{(2)}g)_{\a\b}- 3(A^2)_{\a\b}.
\end{equation}

We define a tensor $T\in \mathcal{T}^{\s 1}_1(\Sigma)$, also known as the shape operator,  by its components as
\[
T^\alpha_\beta:=\frac 12g^{\alpha \gamma}\p_n g_{\gamma \beta}.
\]
Next we calculate using \eqref{eq:mean_curvature} and \eqref{eq:Gammaban}--\eqref{eq:Gamma_squared}
\begin{multline*}
 \Ric_{\ol g}(\p_\alpha,N)=\p_j\Gamma^j_{\alpha n}-\p_\alpha\Gamma^j_{j n}  +(\Gamma^b_{\alpha n}\Gamma^j_{j b}-\Gamma^b_{n j}\Gamma^j_{\alpha b} )  \\
 =\frac 12 \p_j(\ol g^{ja}\p_n\ol g_{\alpha a}) -\frac 12 \p_\alpha(\ol g^{ja}\p_n\ol g_{j a})  +(\Gamma^\beta_{\alpha n}\Gamma^j_{j \beta}-\Gamma^\beta_{n \gamma }\Gamma^\gamma_{\alpha \beta} ) \\
 =\frac 12 \p_\kappa (g^{\kappa \beta }\p_ng_{\alpha \beta }) -\frac 12 \p_\alpha(g^{\beta \gamma}\p_n g_{\beta \gamma})  +(\Gamma^\beta_{\alpha n}\Gamma^\gamma_{\gamma \beta}-\Gamma^\beta_{n \gamma }\Gamma^\gamma_{\alpha \beta} )\\
 =- \nabla_\alpha H+ \p_\kappa T^\kappa_\alpha+ \Gamma_{\gamma\beta}^\gamma T^{\beta}_\alpha- \Gamma_{\alpha\beta}^\gamma T^{\beta}_\gamma \\
 = \nabla_\kappa T^\kappa_\alpha- \nabla_\alpha H,
\end{multline*}
Here $\nabla_\kappa$ and $\nabla_\alpha$
 are covariant derivatives on $(\Sigma,g)$. Especially both two terms on the right are $1$-forms on $\Sigma$ and transform tensorially in changes of coordinates on $\Sigma$.

\subsection{Proof of the curvature recovery proposition}
If $\Sigma$ is an embedded hypersurface in a Riemannian manifold $(\wt M, \ol g )$ of dimension $3$, then the Gauss equation for the Ricci curvature is
\begin{multline}\label{eq:Gauss-Codazzi_Ricci}
  \Ric_{\ol g}(X,Y)=\Ric_{g}(X,Y)+\ol g(\Riem_{\ol g}(X,N)N,Y) \\
  +\sum_{j=1,2}A(X,E_j)A(E_j,Y)-HA(X,Y).
 \end{multline}
 Here $(E_1,E_2)$ is an orthonormal frame of $\Sigma$ and $X$ and $Y$ are tangential to $\Sigma$.  The formula is independent of the choice of the orthonormal frame.

 Let $F:(\Sigma_1,g_1)\to (\Sigma_2,g_2)$ be the isometry given by Theorem \ref{thm:minimal_surface_general}, so that
 \[
  F^*g_2=g_1 \text{ and } F^*A_2=A_1.
 \]
 Let $(E_{2,1}, E_{2,2})$ be an orthonormal frame on $(\Sigma_2,g_2)$. Since $F$ is an isometry, $(E_{1,1},E_{1,2}):=(F^*E_{2,1},F^*E_{2,2} )$ is an orthonormal frame on $(\Sigma_1,g_1)$. (Here $F^*$ acting on vectors is defined as $(F^{-1})_*$ as usual.)

 Let $\p_\a$ and $\p_\b$ be coordinate vectors on $\Sigma_1$. By \eqref{eq:curvature_N_components_witk_k2} and Proposition~\ref{prop:second_normal_jet}, we have
 \begin{multline}\label{eq:Xi_recovered}
  [F^*\Xi_2](\p_\a,\p_\b)=\frac 12 F^*(g_2k^{(2)}_2g_2)(\p_\a,\p_\b)- 3(F^*A_2^2)(\p_\a,\p_\b)\\
  =\frac 12 (g_1k^{(2)}_1g_1)(\p_\a,\p_\b)- 3A_1^2(\p_\a,\p_\b)=\Xi_1(\p_\a,\p_\b).
 \end{multline}
 Here we also used that $F^*(A_2^2)=F^*(A_2g_2^{-1}A_2)=A_1g_1^{-1}A_1=A_1^2$.
 This proves \eqref{eq:Riemann_to_normal}.

Using  \eqref{eq:Gauss-Codazzi_Ricci} and \eqref{eq:Xi_recovered}, we compute
 \begin{multline*}
  [F^*\Ric(\ol g_2)](\p_\alpha,\p_\beta)=[F^*\Ric(g_2)](\p_\alpha,\p_\beta)+[F^*\Xi_2](\p_\a,\p_\b)\\
  +\sum_{j=1,2}F^*(A_2(\p_\alpha,E_{2,j})A_2(E_{2,j},\p_\beta))-(F^*H_2)(F^*A_2)(\p_\alpha,\p_\beta)\\
  =\Ric(g_1)(\p_\alpha,\p_\beta)+\Xi_1(\p_\a,\p_\b) \\
  +\sum_{j=1,2}(F^*A_2)(\p_\alpha,F^{-1}_*E_{2,j})(F^*A_2)(F^{-1}_*E_{2,j},\p_\beta)-H_1A_1(\p_\alpha,\p_\beta)\\
  =\Ric(g_1)(\p_\alpha,\p_\beta)+\ol g_1\big(\Riem(\ol g_1)(\p_\alpha,N_1)N_1,\p_\beta\big) \\
  +\sum_{j=1,2}A_1(\p_\alpha,E_{1,j})A_1(E_{1,j},\p_\beta)-H_1A_1(\p_\alpha,\p_\beta)=\Ric(\ol g_1)(\p_\alpha,\p_\beta).
 \end{multline*}
In addition to the already mentioned facts,  in the first equality we used that   $g_1=F^*g_2$ implies
\[
F^*\Ric(g_2)=\Ric(g_1),
\]
where $\Ric(g_j)$ is the Ricci tensor of $(\Sigma_j,g_j)$  for $j=1,2$. (Above in fact $H_1=H_2=0$.) This proves \eqref{eq:Ricci_to_tangential}.

 By \eqref{eq:q_recovered} and \eqref{eq:q_in_terms_of_curvature} we have
\begin{equation}\label{eq:q_pulled_back}
F^*q_2=q_1,\qquad q_j=-\abs{A_j}_{g_j}^2-\Ric(\ol g_j)(N_j,N_j).
\end{equation}
 The second equation \eqref{eq:Ricci_to_normal} then follows since in addition $F^*\abs{A_2}^2_{g_2}=\abs{A_1}^2_{g_1}$ by Theorem \ref{thm:minimal_surface_general}.

The equation \eqref{eq:mixed_Ricci} follows from the formula $\Ric_{\ol g}(\p_\a,N)=\nabla_\kappa T^\kappa_\a-\nabla_\a H$ derived above, since $F^*T_2=T_1$ and $F^*H_2=H_1$ $(=0)$ by Theorem \ref{thm:minimal_surface_general} and since $F$ is an isometry:
\begin{multline}\label{eq:Theta_calc}
 (F^*\Theta_2)(\p_\alpha)=\frac{\p F^\beta}{\p x^\alpha}\Theta_{2,\beta}|_F=\frac{\p F^\beta}{\p x^\alpha}\Ric(\ol g_2)(N_2,\p_\beta)|_F\\
 =\frac{\p F^\beta}{\p x^\alpha}\left( \nabla_{g_2,\kappa} (T_2)^\kappa_{\beta}- \nabla_{g_2,\beta} H_2\right)\Big|_F \\
 =\frac{\p F^\beta}{\p x^\alpha}\left( F^{-1*}(\nabla_{g_1,\kappa} (T_1)^\kappa)_{\beta}- F^{-1*}(\nabla_{g_1} H_1)_\beta\right)\Big|_F \\
 =\frac{\p F^\beta}{\p x^\alpha}\frac{(\p F^{-1})^\gamma}{\p x^\beta}\left( \nabla_{g_1,\kappa} (T_1)^\kappa_{\gamma}- \nabla_{g_1,\gamma} H_1\right)= \nabla_{g_1,\kappa} (T_1)^\kappa_{\alpha}- \nabla_{g_1,\alpha} H_1 \\
 =\Ric(\ol g_1)(N_1,\p_\alpha)=\Theta_1(\p_\alpha).
\end{multline}
Finally, \eqref{eq:scalar_curvature} follows from \eqref{eq:Ricci_to_tangential} and \eqref{eq:Ricci_to_normal}, since $F^*g_2=g_1$ and
\[
R(\ol g_j)=g_j^{\a\b}\,\Ric(\ol g_j)(\p_\a,\p_\b)+\Ric(\ol g_j)(N_j,N_j), \qquad j=1,2.
\]
This completes the proof of Proposition \ref{prop:slice_curvature}.
\section{Construction of the bulk isometry}\label{sec:bulk_isometry}
 For the foliated theorems we use only the induced metrics, second fundamental forms and Jacobi potentials recovered on the individual leaves. We first glue the normalized leafwise isometries smoothly. The additional curvature recovery of Proposition~\ref{prop:slice_curvature} is not needed in this step.

Let the assumptions of Theorem \ref{thm:main1} or of Theorem \ref{thm:main_local} be satisfied, so that for $j = 1,2$ the sets $\mathcal U_j\subset \wt M_j$ admit foliations $(\Sigma^t_j)_{t\in I}$ by minimal surfaces with $\p\Sigma_1^t=\p\Sigma_2^t\subset\Omega$ and $\Sigma^t_1\cap \Omega = \Sigma_2^t\cap\Omega$, such that \eqref{eq: area perturbations are the same} holds. We write
\[
g_j^t:=\iota_{\Sigma_j^t}^*\ol g_j
\]
for the metrics induced on the leaves, where $\iota_{\Sigma_j^t}:\Sigma_j^t\to \wt M_j$ denotes the inclusion. By Theorem \ref{thm:minimal_surface_general} (recall Remark \ref{rmk:leaves_nondegenerate}), for each $t\in I$ there exists a diffeomorphism $F_t: \Sigma_1^t \to \Sigma_2^t$ which is the identity on $\partial \Sigma_1^t = \partial \Sigma_2^t$ and satisfies
    \begin{eqnarray}\label{eq: Ft is isometry}
    g_1^t = F_t^*g_2^t.
    \end{eqnarray}
    Note that this means that for each $t\in I$ and $u\in C^\infty(\Sigma_1^t)$,
\begin{eqnarray}\label{eq: harmonic means harmonic}
\Delta_{g_1^t} u= 0 \iff \Delta_{F_t^*g_2^t} u = 0.
\end{eqnarray}
We remark that $F_t$ is uniquely determined by these properties: if $F_t'$ is another such map, then $F_t'^{-1}\circ F_t$ is an isometry of the connected surface $(\Sigma_1^t,g_1^t)$ which is the identity on $\p\Sigma_1^t$. Its differential at a boundary point is then the identity (it fixes the tangent space of the boundary and the inward unit normal vector of the boundary in $\Sigma_1^t$), and hence $F_t'^{-1}\circ F_t=\Id$.
  We define
		\[
		 \ol F: \mathcal U_1=\bigcup\limits_{t\in I} \Sigma_1^t\to \mathcal U_2=\bigcup\limits_{t\in I} \Sigma_2^t
		\]
by
\begin{eqnarray}\label{eq: def of ol F}
 \ol F(p):=F_t(p),\qquad p\in\Sigma_1^t.
\end{eqnarray}
Since the leaves are pairwise disjoint and each $F_t$ is a bijection, $\ol F$ is a well-defined bijection which maps $\Sigma_1^t$ onto $\Sigma_2^t$ for every $t\in I$. We now show that $\ol F$ is the identity near the boundaries of the leaves, and that it is a diffeomorphism.

\begin{lem}\label{lem: isometry is identity}
There is an open set $\mathcal V\subset \wt M_1$ with $\bigcup_{t\in I}\p\Sigma_1^t\subset \mathcal V$ and $\mathcal V\cap \mathcal U_1\subset \Omega$ such that
\[
\ol F(p)=p\quad \text{for all } p\in \mathcal V\cap \mathcal U_1.
\]
Moreover, for every $t\in I$, the map $F_t$ is the identity on each connected component of $\Sigma_1^t\cap\Omega$ which meets $\p\Sigma_1^t$.
\end{lem}
\begin{proof}
Fix \(t\in I\). The two leaves and their induced metrics agree in \(\Omega\).
At \(q\in\p\Sigma_1^t\), the boundary-fixing isometry \(F_t\) fixes the tangent
to the boundary and maps the inward unit normal vector \(\nu_t(q)\) of \(\p\Sigma_1^t\) in \(\Sigma_1^t\) to the inward unit normal vector of \(\p\Sigma_2^t\) in \(\Sigma_2^t\).
These normal vectors coincide in the common exterior, so \(dF_t(q)=\Id\).
For a sufficiently short inward normal geodesic
\[
 \gamma_q(s)=\exp_q^{g_1^t}(s\nu_t(q)),\qquad s\geq0,
\]
the entire segment lies in \(\Omega\). Its image under \(F_t\) is a geodesic
with the same initial point and velocity. The metrics coincide on this segment,
so uniqueness of geodesics gives \(F_t(\gamma_q(s))=\gamma_q(s)\).

This identity holds on a jointly open collar, not just separately on individual
leaves. To verify this, let
\(B=\bigcup_{t\in I}\p\Sigma_1^t\), which is the smooth boundary of the foliated
region. On a small patch of \(B\), consider
\[
 E(q,s)=\exp_q^{g_1^{t(q)}}(s\nu_{t(q)}(q)).
\]
The smooth foliation up to the leaf boundaries implies that \(E\) is smooth in
\((q,s)\). At \(s=0\), its differential maps tangent vectors to \(B\) identically
and the \(s\)-direction to \(\nu_{t(q)}(q)\), which is transverse to \(B\).
The inverse function theorem therefore gives a relative neighbourhood of this
boundary patch in \(\mathcal U_1\). Shrinking it makes it a subset of \(\Omega\),
and the preceding geodesic argument gives \(\ol F=\Id\) there. Taking the union
over boundary patches gives a relatively open set \(\mathcal C\subset\mathcal U_1\)
containing \(B\), with \(\mathcal C\subset\Omega\) and \(\ol F=\Id\) on
\(\mathcal C\). By the subspace topology, choose an ambient open set
\(\mathcal V\) with \(\mathcal V\cap\mathcal U_1=\mathcal C\).

Finally, let \(D\) be a connected component of \(\Sigma_1^t\cap\Omega\)
meeting \(\p\Sigma_1^t\). The maps \(F_t|_D\) and the inclusion into
\(\Sigma_2^t\) are local isometries and agree in a nonempty collar patch.
In the interior of \(D\), the set where they agree in value and differential is
closed and open, again by the exponential-map identity. The interior of this
connected surface with boundary is connected, so they agree there and, by
continuity, on all of \(D\).
\end{proof}

\begin{lem}\label{lem: F is diffeo}
The map
$\ol F: \mathcal U_1\to \mathcal U_2$
defined by \eqref{eq: def of ol F} is a diffeomorphism.
\end{lem}
\begin{proof}
The same construction applies to \(F_t^{-1}\), so it is enough to prove that
\(\ol F\) is smooth. Fix \(t_0\in I\) and a point \(p\in\Sigma_1^{t_0}\).
Near the boundary of that leaf, smoothness follows from the jointly open collar
in Lemma~\ref{lem: isometry is identity}. Choose a path in the interior of the
connected leaf from this collar to \(p\), unless \(p\) is already in the collar.

We propagate smoothness along this path. Suppose it is known near a smooth
section \(p(t)\in\Sigma_1^t\) through a point already reached. Choose source
and target foliation charts independently, around \(p(t_0)\) and
\(F_{t_0}(p(t_0))\). Then \(c(t)=F_t(p(t))\) and
\(L_t=dF_t|_{T_{p(t)}\Sigma_1^t}\) depend smoothly on \(t\), by the already
known smoothness. Choose smoothly varying orthonormal frames at \(p(t)\) and
\(c(t)\). In these frames \(L_t\) is an orthogonal matrix. For a sufficiently
small fixed \(r>0\) and \(t\) in a sufficiently small interval about \(t_0\),
the exponential maps at these two sections are defined on tangent balls of
radius \(r\), stay in the respective interior charts, and are diffeomorphisms
onto their images. These properties follow by shrinking the balls at \(t_0\)
and using smooth dependence of the geodesic equations.

Isometries map geodesics to geodesics, so
\begin{equation}\label{eq: F via exp}
 F_t\bigl(\exp_{p(t)}^{g_1^t}v\bigr)
       =\exp_{c(t)}^{g_2^t}(L_tv),\qquad |v|_{g_1^t}<r.
\end{equation}
The right-hand side is jointly smooth in \((t,v)\), and the map
\((t,v)\mapsto\exp_{p(t)}^{g_1^t}v\) is a foliation coordinate chart after
restriction. Thus \eqref{eq: F via exp} proves smoothness on the next such
neighbourhood. There is no assumption here that the sets \(F_t(U_t)\) already
vary smoothly.

The chosen path is compact. Choose finitely many sufficiently short geodesic
neighbourhoods along it, with the centre of each successive neighbourhood in
its predecessor. The corresponding target charts are chosen along its image
under the fixed isometry \(F_{t_0}\). Applying the preceding step finitely many
times, and shrinking the interval in \(t\) when necessary, proves smoothness
near \(p\). This proves smoothness everywhere, including the boundary collar.
The same proof for \(F_t^{-1}\) gives a smooth inverse. The union map is already
bijective because the leaves are disjoint and each \(F_t\) is bijective, so
\(\ol F\) is a diffeomorphism.
\end{proof}

Thanks to Lemma \ref{lem: F is diffeo}, $\overline F$ is a diffeomorphism, so we can speak of pull-backs and push-forwards by $\ol F$. Since $\ol F$ maps each leaf $\Sigma_1^t$ onto $\Sigma_2^t$, its differential maps tangent vectors of $\Sigma_1^t$ to tangent vectors of $\Sigma^t_2$.
We also denote
\begin{equation}\label{eq:both_metrics_on_side1}
 \widetilde g:=\ol F^*\og_2 \text{ and } \ol g:=\ol g_1,
\end{equation}
which are both Riemannian metrics on $\mathcal U_1$. By construction, $\ol F:(\mathcal U_1,\wt g)\to (\mathcal U_2,\ol g_2)$ is an isometry, and by \eqref{eq: Ft is isometry} the metrics $\wt g$ and $\ol g$ induce the same metric $g_1^t$ on each leaf $\Sigma_1^t$. Moreover, since $\ol F(\Sigma_1^t)=\Sigma_2^t$ is a $\ol g_2$-minimal surface, each $\Sigma_1^t$ is a $\wt g$-minimal surface. The remaining task is to show that $\wt g=\ol g$.

	\section{Lapse and shift equations and proofs of the foliated theorems}\label{sec:gauging}
The coordinate formulas and geometric identities in Sections~5.1--5.2 and
Appendix~\ref{sect:appendix} hold in ambient dimension \(n+1\), with
\(n\)-dimensional minimal hypersurfaces. We use this notation there and set
\(n=2\) when applying the identities to the foliated recovery theorems.
The two ambient metrics constructed in Section~\ref{sec:bulk_isometry} already
induce the same metric on each leaf. It remains to identify the lapse and shift.
Their equations have coefficients determined by the induced metric, the second
fundamental form and the Jacobi potential. Thus Theorems~\ref{thm:main1}
and~\ref{thm:main_local} do not use Proposition~\ref{prop:second_normal_jet}
or the recovery of the full ambient Ricci tensor.

%We retain the proof in the $H$-harmonic coordinates of
%\cite{anderson2008unique,anderson2010erratum}. The special case here requires
%only local coordinates harmonic on the leaves of an already given minimal
%foliation.

\subsection{Foliation harmonic coordinates}\label{subsect:harmonic coordinates}

For the coordinate formulas, let \(\bar g\) and \(\tilde g\) be two metrics
on an \((n+1)\)-dimensional region with the same foliation by minimal
hypersurfaces and the same induced metric on each leaf. In the application,
\(\bar g=\bar g_1\) and \(\tilde g=\bar F^*\bar g_2\) are the metrics from
Section~\ref{sec:bulk_isometry}, and \(n=2\).

We will work in (local) foliation coordinate systems $(t,x^1,\ldots,x^n)$ which are given by harmonic coordinates on each slice $\Sigma^t$. In such coordinates, the symmetric tensors $\tilde g$ and $\bar g$ are given by
\begin{eqnarray}\label{eq: overline g in coordinates}
\overline g = (\overline u^{\s 2} +\overline g_{\alpha\beta} \overline \sigma^\alpha\overline\sigma^\beta ) dt^2 + 2\,\overline g_{\alpha\beta} \overline\sigma^\beta dx^\alpha dt + \overline g_{\alpha\beta} dx^\alpha dx^\beta
\end{eqnarray}
and
\begin{eqnarray}\label{eq: tilde g in coordinates}
    \widetilde g = (\widetilde u^{\s 2} + \widetilde g_{\alpha\beta}\widetilde \sigma^\alpha\widetilde\sigma^\beta) dt^2 + 2\,\widetilde g_{\alpha\beta}\widetilde\sigma^\beta  dx^\alpha dt + \widetilde g_{\alpha\beta} dx^\alpha dx^\beta,
\end{eqnarray}
where the Greek indices $\alpha, \beta$ run over $1,\ldots,n$. Here $\ol u,\wt u>0$ are the lapse functions and $\ol\sigma=\ol\sigma^\a\p_\a$, $\wt\sigma=\wt\sigma^\a\p_\a$ the shift vector fields of the two metrics with respect to the foliation, see Appendix \ref{sect:appendix}. The induced leaf metrics agree, so
\begin{eqnarray}\label{eq: metrics equal on slice}
\tilde g_{\alpha\beta} = \bar g_{\alpha\beta}.
\end{eqnarray}
So it only remains to show that $\bar u=\tilde u$ and $\tilde\sigma = \bar\sigma$.

	\subsection{Equations for shift and lapse functions in H-harmonic coordinates}
In this section we derive a system of equations for the lapse and shift
functions in the \(H\)-harmonic coordinates introduced in
\cite{anderson2008unique,anderson2010erratum}. Under the assumption that the
region is foliated by minimal surfaces, we also construct such coordinates.  As defined there, a foliation coordinate system $(t,x)$ is $H$-harmonic if $\{x(t,\cdot)^\a\}$ are harmonic coordinates for $\Delta_{\iota_{\Sigma^t}^*\bar g}$ on $\Sigma^t$, and each $\Sigma^t$ has fixed prescribed mean curvature $H$. We consider the case $H=0$, which means that $H$-harmonic coordinates are given by a (local) foliation by minimal hypersurfaces $(\Sigma^t,g_t)$ together with harmonic coordinates $x_t^\a$ on each of them.
The harmonic-coordinate condition is equivalent to $g_t^{\mu\nu}\Gamma^{\alpha}_{\mu\nu}(g_t)=0$ for every tangential index $\alpha$ and every $t$.
\begin{lem}\label{lem: harmonic H-coordinates}
    Assume that $(M,\ol g)$ is an $(n+1)$-dimensional Riemannian manifold. Let $p\in M$ and assume that $p$ is contained in a neighbourhood foliated by minimal hypersurfaces. Then there exist $H$-harmonic coordinates on a neighbourhood of $p$, where each slice has $H=0$. The coordinate neighbourhood can be chosen so that its intersection with each slice is connected.
\end{lem}
\begin{proof}
By assumption, we have coordinates $(t,x)$, for $t$ small and $x\in B(0,\delta)\subset \R^n$ for some $\delta>0$, where each slice $\{t=\text{const}\}$ is a minimal hypersurface and thus has mean curvature $H=0$.

We now need to construct harmonic coordinates on each leaf of the foliation around the point $p\in M$. To this end, assume that $p\in \Sigma^0$ and let $(u_0^1, \dots, u_0^n)$ be harmonic functions with respect to $\Delta_{ \iota_{\Sigma^0}^*\bar g}$ defined in $B_\delta(p) \subset \Sigma^0$ with linearly independent gradient at $p$.

By definition of foliation, for each $t\in \mathbb R$ sufficiently close to $0$, there is a diffeomorphism $\varphi_t : B_\delta(p) \to \varphi_t(B_\delta(p)) \subset \Sigma^t$ with $\varphi_0(\cdot) = {\rm Id}$ and $$\mathbb R\times B_\delta(p) \ni (t,q)\mapsto \varphi_t(q) \in M$$ is smooth.

Now consider for each $\alpha = 1,\dots, n$ and $t\in \mathbb R$ small, the BVP
\begin{eqnarray}\label{eq: laplace eq for coord funct} \Delta_{\varphi_t^* \iota_{\Sigma^t}^*\bar g} v_t^\alpha = 0,\ {\rm in}\ B_\delta(p),\ v_t^\alpha\mid_{\partial B_\delta(p)} = u_0^\alpha \mid_{\partial B_\delta(p)}.
\end{eqnarray}

The solution $v_t^\alpha(q)$ to this BVP is jointly smooth in $(t,q)$ thanks to smooth dependence of solutions on the parameters. Therefore, if we define
$$u_t^\alpha(q) := v_t^\alpha(\varphi_t^{-1}(q)) \in C^\infty(\varphi_t(B_\delta(p)))$$ for $\alpha = 1,\dots, n$, then $(u^1_t,\dots, u^n_t)$ is a vector of functions jointly smooth in $(t,q)$, and for each fixed small $t\in \mathbb R$ these functions have linearly independent gradients at $\varphi_t(p)\in \Sigma^t$. Therefore, for each $t\in\mathbb R$ sufficiently small, $(u^1_t,\dots, u^n_t)$ is a coordinate system for $\Sigma^t$ in a small neighbourhood of $\varphi_t(p)$. Furthermore, for each small $t\in \mathbb R$, $u^\alpha_t(\cdot)$ is harmonic with respect to $\iota_{\Sigma^t}^*\bar g$ thanks to \eqref{eq: laplace eq for coord funct}. Choose a smaller ball compactly contained in the domain where $(u_0^1,\ldots,u_0^n)$ is a coordinate chart, and restrict the already constructed solutions to that ball. Smooth dependence and openness of the embedding condition on its compact closure show that, after reducing the time interval, $(u_t^1,\ldots,u_t^n)$ is a coordinate system there for each $t$. Relabel this smaller ball and time interval for the notation below.

Let us then define
\[ U:
\bigcup_{t\in(-\delta,\delta)} \varphi_t(B_\delta(p))\to \R^{n+1}, \quad U(q)=(t,u^1_t(q),\ldots,u^n_t(q)) \quad\text{for } q\in \varphi_t(B_\delta(p)).
\]
In the coordinates $(t,x)$, the Jacobian of this mapping has the form
\[
\begin{bmatrix}
    1       & 0  \\
    \#       & D_x(u_t^1,\ldots,u_t^n)
\end{bmatrix}
\]
and thus is invertible. Consequently $U$ is a coordinate system on a neighbourhood of $p$, where for each $t$, $(u_t^1,\ldots,u_t^n)$ are harmonic coordinates for the metric on the constant $t$-slices. Since the mean curvatures of the slices are independent of the chosen coordinates of the slices, each constant $t$-slice has zero mean curvature. Finally, restricting $U$ to the preimage of a product of an interval and a ball, we obtain a coordinate neighbourhood whose intersection with each slice is connected.
\end{proof}

The next proposition gives a system of equations for the lapse and shift functions in $H$-harmonic coordinates. We consider the setting where $(\Sigma^t,g_t)$, $g_t=\iota_{\Sigma^t}^*\overline g$, is a minimal hypersurface of $(\overline M, \overline g)$ for $t\in (-1,1)$, and the slices $\Sigma^t$ provide a local foliation of $\overline M$. Let $A = A(x,t)$ be the second fundamental form of the slice $\Sigma^t$, which has the coordinate expression $\left(A_{\alpha\beta}\right)_{\alpha,\beta = 1}^n$.
\begin{prop}\label{prop: elliptic equation for lapse and shift}
 Assume that $(t,x^\a)$ are $H$-harmonic coordinates, where each slice $\Sigma^t$ has $H=0$, and
 in which the metric $\bar g$ has the expression given by \eqref{eq: overline g in coordinates}. Then the lapse $u$ and the shift $\sigma = (\sigma^\alpha)_{\alpha =1,\dots, n}$ of these coordinates satisfy on each slice $\Sigma^t$
\begin{equation}\label{eq:lapseshifteq_simple}
\Delta_{g_t}u+\left(\abs{ A}_{g_t}^2+\Ric_{\overline g}(N,N)\right)u=0
 \end{equation}
 and for each $\alpha =1,\dots, n$
 \begin{eqnarray}\label{eq: elliptic PDE for sigma}
 \Delta_{g_t}\sigma^\alpha
=-2u\langle A,\nabla^2x^\alpha\rangle_{g_t}
-2A(\nabla u,\nabla x^\alpha)
-2u(\divergence_{g_t}A)(\nabla x^\alpha).
\end{eqnarray}
Here $\nabla$, $\Delta_{g_t}$ and $\divergence_{g_t}$ are taken on the slice $(\Sigma^t,g_t)$, and the unit normal $N$ of $\Sigma^t$ is given by
\begin{eqnarray}
    \label{eq: normal of slice}
    N = u^{-1} (\partial_t -\sigma^\alpha \partial_\alpha).
\end{eqnarray}
\end{prop}
We will prove this proposition in Appendix \ref{sect:appendix}. We remark that if $(\tilde u, \tilde \sigma^\a)$ is another pair of solutions to \eqref{eq:lapseshifteq_simple}--\eqref{eq: elliptic PDE for sigma} on a slice, with the same coefficients $g_t$, $A$ and $\Ric_{\ol g}(N,N)$, which agrees with \( (u, \sigma^\alpha) \)  on an open subset of the slice, then $u=\tilde u$  and $\sigma^\a=\tilde \sigma^\a$  on the whole (connected) slice. This follows from the unique continuation principle (UCP). Indeed, the first equation and UCP imply \( u \equiv \tilde u \), and then, with \( u - \tilde u \) identically zero, the second equation and UCP force \( \sigma^\alpha-\tilde \sigma^\a = 0 \). We also note that the first equation says that the lapse function of any foliation by minimal surfaces satisfies the linearized minimal surface equation (also known as the stability or Jacobi equation). This is a well-known fact, see for example \cite{Neves}.

We now apply these identities with \(n=2\) in the setting of
Section~\ref{sec:bulk_isometry}.
Let $N_1$ and $N_2$ be the smooth unit vector fields on $(\mathcal U_1,\overline g_1)$ and $(\mathcal U_2, \overline g_2)$ respectively which are normal to the leaves $\{\Sigma_1^t\}_{t\in I}$ and $\{\Sigma_2^t\}_{t\in I}$ and point in the direction of increasing $t$. Then $N_1=N_2$ on $\Sigma_1^t\cap\Omega=\Sigma_2^t\cap\Omega$, and we compute the second fundamental forms $A_1$ and $A_2$ of the leaves (in particular those in Theorem \ref{thm:minimal_surface_general}) with respect to these normals. Denote by $\widetilde N := \left(\ol F^{-1}\right)_* N_2$ the $\widetilde g$-unit normal to the leaves $\Sigma^t_1$ (recall that $\ol F:(\mathcal U_1,\wt g)\to (\mathcal U_2,\ol g_2)$ is an isometry mapping $\Sigma_1^t$ onto $\Sigma_2^t$), and set $ N:= N_1$.

\begin{lem}
\label{lem: tilde g is same as barg}
In the coordinate expressions \eqref{eq: tilde g in coordinates} and \eqref{eq: overline g in coordinates} of $\tilde g$ and $\bar g$ we have
$\tilde u = \bar u$ and $\tilde\sigma = \bar \sigma$. That is, $\tilde g = \bar g$ as symmetric $2$-tensors on $\mathcal U_1=\bigcup_t \Sigma^t_1$.
\end{lem}
\begin{proof}
    Note that since for each $t\in I$, $F_t : \Sigma_1^t \to \Sigma_2^t$ is an isometry from $(\Sigma_1^t, g_1^t)$ to $(\Sigma_2^t, g_2^t)$, we already have that
    \begin{eqnarray}\label{eq: pullback are the same}
    \iota_{\Sigma_1^t}^* \tilde g = \iota_{\Sigma_1^t}^* \overline g=g_1^t.
    \end{eqnarray}

    Fix $\hat t\in I$, which we assume without loss of generality to be $\hat t = 0$, and write $\Sigma_1:=\Sigma_1^0$. Define $\mathcal O\subset \Sigma_1$ to be
    \begin{multline}
        \label{eq:identifying set}
        \mathcal O := \{ p \in \Sigma_1 \mid \exists U\subset \Sigma_1\ {\rm open}\ {\rm containing}\ p\ {\rm s.t.}\\ \tilde g(q) = \overline g(q) \text{ as bilinear forms on } T_q\wt M_1 \text{ for all } q\in U\}.
    \end{multline}
    By Lemma \ref{lem: isometry is identity} we have $\ol F=\Id$ on the open set $\mathcal V\cap\mathcal U_1\subset \Omega$, and since $\ol g_1=\ol g_2$ in $\Omega$, it follows that $\tilde g=\ol F^*\ol g_2 = \overline g$ on $\mathcal V\cap\mathcal U_1$. As $\mathcal V\cap\mathcal U_1$ contains a neighbourhood of $\p\Sigma_1$ in $\mathcal U_1$, the set $\mathcal O$ is non-empty. We claim that
    \begin{eqnarray}
        \label{eq: O is everything}
        \mathcal O = \Sigma_1 .
    \end{eqnarray}
By definition, the set $\mathcal O$ is an open subset of $\Sigma_1$. We now need to check that it is also a closed subset. To this end, let $\hat p\in \Sigma_1$ be an accumulation point of $\mathcal O$. The full boundary collar already lies in $\mathcal O$, so it suffices to consider an interior point $\hat p$. Let $\mathcal U\subset \mathcal U_1$ be an open subset containing $\hat p$ on which we can construct the foliation $H$-harmonic (w.r.t. $\bar g$) coordinates
$p\mapsto (t(p), x^1(p), x^2(p))$ described in Lemma \ref{lem: harmonic H-coordinates}, where $\{t(p) = s\} \subset \Sigma_1^s$ and on each slice $\Sigma_1^t$, $p\mapsto (x^1(p), x^2(p))$ forms a $g_1^t$-harmonic coordinate system on the slice. By Lemma \ref{lem: harmonic H-coordinates} we may assume that $U:= \mathcal U \cap \Sigma_1$ is connected. By \eqref{eq: harmonic means harmonic} and \eqref{eq: pullback are the same}, on each slice $\Sigma^t_1$, the functions $x^1(p)$ and $x^2(p)$ are also $\iota_{\Sigma_1^t}^*\widetilde g$-harmonic functions. Since each slice $\Sigma_1^t$ is a $\bar g$-minimal surface, its mean curvature with respect to $\bar g$ vanishes identically. Furthermore, since $\overline F(\Sigma_1^t) = \Sigma_2^t$ are $\bar g_2$-minimal surfaces and $\ol F:(\mathcal U_1,\wt g)\to(\mathcal U_2,\ol g_2)$ is an isometry, the slices $\Sigma_1^t$ are $\tilde g$-minimal surfaces, so their mean curvature with respect to $\tilde g$ vanishes as well. As such,

\begin{equation}\label{eq: both are H-harmonic coordinates with H=0}
(t,x^1,x^2)\ \text{is $H$-harmonic with $H=0$ for both $\bar g$ and $\tilde g$.}
\end{equation}

As $\hat p\in \Sigma_1$ is an accumulation point of $\mathcal O$, we have $U\cap \mathcal O\neq \emptyset$, and hence there is a nonempty open set $V\subset U$ on which $\widetilde g = \overline g$. Using the coordinate expressions \eqref{eq: tilde g in coordinates} and \eqref{eq: overline g in coordinates}, we have
\begin{eqnarray}\label{eq: agree in some nbhd}
\widetilde u = \overline u\ {\rm and}\  \widetilde \sigma = \overline\sigma \quad {\rm in}\ V.
\end{eqnarray}
 By \eqref{eq: pullback are the same},
\begin{eqnarray}\label{eq:lower right corner agree}
    \widetilde g_{\alpha\beta}= \overline g_{\alpha\beta}\quad {\rm in}\ \mathcal U.
\end{eqnarray}

By \eqref{eq:lapseshifteq_simple}, the function $\overline u$ satisfies
$$\Delta_{g_1^0}\overline u+\abs{\overline A}_{g_1^0}^2\overline u+\overline{ \Ric}(N,N)\overline u=0\quad {\rm in}\ U,$$
where $\overline{ \Ric}$ and $\overline A$ are the Ricci curvature and the second fundamental form of $\Sigma_1$ with respect to $\overline g$.
 By \eqref{eq: both are H-harmonic coordinates with H=0}, we have that $\widetilde u$ also satisfies
$$  \Delta_{g_1^0}\widetilde u+\abs{\widetilde A}_{g_1^0}^2\widetilde u+\widetilde{ \Ric}(\widetilde N,\widetilde N)\widetilde u=0\quad {\rm in}\ U,$$
where $\widetilde{ \Ric}$ and $\widetilde A$ are the Ricci curvature and the second fundamental form of $\Sigma_1$ with respect to $\widetilde g$. We claim that
\begin{eqnarray}\label{eq: second fundamental forms agree}
 \abs{\overline A}_{g_1^0}^2+\overline{ \Ric}(N,N) = \abs{\widetilde A}_{g_1^0}^2+\widetilde{ \Ric}(\widetilde N,\widetilde N),\quad{\rm and\ } \quad  \overline A = \widetilde A \quad \text{on } \Sigma_1.
\end{eqnarray}
Indeed, since $\ol F:(\mathcal U_1,\wt g)\to (\mathcal U_2,\ol g_2)$ is an isometry mapping $\Sigma_1$ onto $\Sigma_2:=\Sigma_2^0$ with $\ol F|_{\Sigma_1}=F_0$, we have $\widetilde A=F_0^*A_2$ and $\abs{\widetilde A}^2_{g_1^0}+\widetilde{\Ric}(\widetilde N,\widetilde N)=F_0^*\big(\abs{A_2}^2_{g_2^0}+\Ric(\ol g_2)(N_2,N_2)\big)$, where $A_2$ is the second fundamental form of $\Sigma_2\subset (\wt M_2,\ol g_2)$. By Theorem \ref{thm:minimal_surface_general}, $F_0^*A_2=A_1=\ol A$, and by \eqref{eq:q_pulled_back}, $F_0^*\big(\abs{A_2}^2_{g_2^0}+\Ric(\ol g_2)(N_2,N_2)\big)=-F_0^*q_2=-q_1=\abs{\ol A}^2_{g_1^0}+\ol{\Ric}(N,N)$. This proves \eqref{eq: second fundamental forms agree}.

Combined with \eqref{eq: pullback are the same}, we can subtract the two equations for $\overline u$ and $\widetilde u$ to yield
\begin{eqnarray}\label{eq:elliptic PDE for lapse}
\Delta_{g_1^0}(\overline u - \widetilde u)+\left(\abs{\overline A}_{g_1^0}^2+\overline{ \Ric}(N,N)\right)(\overline u-\widetilde u)=0 \quad {\rm in}\ U.
\end{eqnarray}

Now by \eqref{eq: agree in some nbhd}, $\overline u = \widetilde u$ in $V$, so by unique continuation (recall that $U$ is connected) \eqref{eq:elliptic PDE for lapse} implies that
\begin{eqnarray}\label{eq: lapse agree}
u:=  \overline u = \widetilde u \quad {\rm in} \ U.
\end{eqnarray}
Now by \eqref{eq: elliptic PDE for sigma}, we have that the following equations for $\overline \sigma^\alpha$ and $\widetilde \sigma^\alpha$ hold in $U$:
\begin{eqnarray}\label{eq: shift overline}
\quad\Delta_{g_1^0}  \overline\sigma^\a =  -2u\langle \overline A,\nabla^2x^\alpha\rangle_{g_1^0}
-2\overline A(\nabla u,\nabla x^\alpha)
-2u(\divergence_{g_1^0}\overline A)(\nabla x^\alpha),
\end{eqnarray}
\begin{eqnarray}\label{eq: shift tilde}
\quad \Delta_{g_1^0}  \widetilde\sigma^\a =  -2u\langle \widetilde A,\nabla^2x^\alpha\rangle_{g_1^0}
-2\widetilde A(\nabla u,\nabla x^\alpha)
-2u(\divergence_{g_1^0}\widetilde A)(\nabla x^\alpha).
\end{eqnarray}
Note that by \eqref{eq: pullback are the same}, the connection $\nabla$ is the Levi-Civita connection of $g_1^0=\iota_{\Sigma_1}^*\widetilde g= \iota_{\Sigma_1}^*\overline g$, and we have used \eqref{eq: lapse agree}. Furthermore, by \eqref{eq: second fundamental forms agree}, we have that $\widetilde A = \overline A$. Therefore, subtracting \eqref{eq: shift overline} from \eqref{eq: shift tilde}, we have that $\overline \sigma^\alpha- \widetilde \sigma^\alpha$ is harmonic in $U$.

Combined with \eqref{eq: agree in some nbhd}, unique continuation then gives that $\overline\sigma = \widetilde\sigma$ in $U$. Combined with \eqref{eq: lapse agree} and \eqref{eq:lower right corner agree}, we have that $\overline g = \widetilde g$ at every point of $U$, since the metrics have the coordinate expressions given by \eqref{eq: overline g in coordinates} and \eqref{eq: tilde g in coordinates}.

By definition, this means that $\hat p \in \mathcal O$, and we have that $\mathcal O$ is a non-empty open and closed subset of the connected surface $\Sigma_1$. So \eqref{eq: O is everything} holds. Since $\hat t\in I$ was arbitrary, the proof is complete.
\end{proof}

\begin{proof}[Proof of Theorems \ref{thm:main1} and \ref{thm:main_local}]
Let the assumptions of Theorem \ref{thm:main1} or of Theorem \ref{thm:main_local} hold, and let $\ol F:\mathcal U_1\to\mathcal U_2$ be the map constructed in Section \ref{sec:bulk_isometry}. By Lemma \ref{lem: F is diffeo}, $\ol F$ is a diffeomorphism, and by Lemma \ref{lem: tilde g is same as barg}, $\ol F^*\ol g_2=\wt g=\ol g=\ol g_1$. Thus $\ol F:(\mathcal U_1,\ol g_1)\to (\mathcal U_2,\ol g_2)$ is an isometry. By construction $\ol F$ maps $\Sigma_1^t$ onto $\Sigma_2^t$ for every $t\in I$, and by Lemma \ref{lem: isometry is identity} it is the identity on the neighbourhood $\mathcal V\cap \mathcal U_1\subset \Omega$ of $\bigcup_{t\in I}\p\Sigma_1^t$. This proves Theorem \ref{thm:main_local} and the first part of Theorem \ref{thm:main1}.

Assume finally that every connected component of $\Sigma_1^t\cap\Omega$ meets $\p\Sigma_1^t$ for every $t\in I$. Then by the last statement of Lemma \ref{lem: isometry is identity}, $F_t$ is the identity on $\Sigma_1^t\cap \Omega$ for every $t$, that is, $\ol F$ is the identity on $\mathcal U_1\cap\Omega=\mathcal U_2\cap \Omega$. Consequently
\[
\ol F(M_1)=\ol F(\mathcal U_1\setminus\Omega)=\mathcal U_2\setminus \Omega=M_2,
\]
so that $\ol F$ restricts to an isometry $(M_1,\ol g_1)\to (M_2,\ol g_2)$. This completes the proof of Theorem \ref{thm:main1}.
\end{proof}

\section{Einstein rigidity}\label{sec:einstein}

    \begin{proof}[Proof of Theorem \ref{thm:main2}]
Let $x_0\in M$. Since $(\widetilde M, \bar g_1)$ is non-trapping, there is a non-degenerate embedded minimal surface $\Sigma_1\subset \widetilde M$ of $(\wt M,\ol g_1)$ that contains $x_0$ and has $\p\Sigma_1\subset\Omega$. By the assumption $\mathcal A_1=\mathcal A_2$, there is a non-degenerate embedded minimal surface $\Sigma_2\subset \widetilde M$ of $(\wt M,\ol g_2)$ with $\Sigma_1\cap\Omega=\Sigma_2\cap\Omega$ whose area data agree with those of $\Sigma_1$.

By Proposition \ref{prop:slice_curvature}, there exists an isometry $F:(\Sigma_1,g_1)\to (\Sigma_2,g_2)$ of the induced metrics such that for all $X,Y\in T\Sigma_1$
 \begin{align*}
  \big(F^*\Ric(\ol g_2)\big)(X,Y)&=\Ric (\ol g_1)(X,Y), \\
  F^*\big(\Ric (\ol g_2)(N_2,N_2)\big)&=\Ric (\ol g_1)(N_1,N_1),  \\
  (F^*\Theta_2)(X)&=\Theta_1(X),\\
  F^*(R(\ol g_2))&=R(\ol g_1).
 \end{align*}
Here $N_1$ and $N_2$ are the unit normal vector fields of $\Sigma_1$ and $\Sigma_2$, and $\Theta_j(X)=\Ric(\ol g_j)(N_j,X)$ for $X\in T\Sigma_j$. Let $n=3$. Since $\ol g_2$ is Einstein, $\Ric(\ol g_2)=\frac{R(\ol g_2)}{n}\ol g_2$. Thus the normal component of the trace-free Ricci tensor of $\ol g_1$ at $x_0$ satisfies
\begin{multline*}
0=F^*\left(\Ric(\ol g_2)(N_2,N_2)-\frac{R(\ol g_2)}{n}\ol g_2(N_2,N_2)\right) \\
=\Ric(\ol g_1)(N_1,N_1)-\frac{R(\ol g_1)}{n}\ol g_1(N_1,N_1).
\end{multline*}
Here we used that $\ol g_1(N_1,N_1)=\ol g_2(N_2,N_2)=1$.

Now let $X,Y\in T\Sigma_1$ and denote $\widetilde X = F_*X\in T\Sigma_2$ and $\widetilde Y = F_*Y\in T\Sigma_2$. We compute
\begin{multline*}
0 = F^*\left(\Ric(\bar g_2)(\widetilde X,\widetilde Y)-\frac{R(\bar g_2)}{n}\bar g_2(\widetilde X,\widetilde Y)\right) \\
= (F^*\Ric(\bar g_2))(X,Y) - \frac{F^*(R(\bar g_2))}{n} F^*(\bar g_2(\widetilde X,\widetilde Y)) \\
= \Ric(\bar g_1)(X,Y) - \frac{R(\bar g_1)}{n} \bar g_1(X,Y),
\end{multline*}
where we have used that $F^*(\bar g_2(\widetilde X,\widetilde Y)) = \bar g_1(X,Y)$ since $F$ is an isometry between the induced metrics on $\Sigma_1$ and $\Sigma_2$. Finally, the mixed components vanish as well: since $\ol g_2(N_2,\widetilde X)=0$, the Einstein condition gives $\Theta_2=0$, and hence
\[
\Ric(\bar g_1)(N_1,X)-\frac{R(\bar g_1)}{n}\bar g_1(N_1,X)=\Theta_1(X)=(F^*\Theta_2)(X)=0.
\]

Since $T_{x_0}\wt M=T_{x_0}\Sigma_1\oplus \R N_1$, the three identities above show that
\[
\Ric(\bar g_1)-\frac{R(\bar g_1)}{n}\bar g_1 = 0
\]
at $x_0$. Since $x_0\in M$ was arbitrary, $\bar g_1$ is Einstein on $M$. In the exterior it equals the Einstein metric $\bar g_2$, so it is Einstein on $\widetilde M$. The two Einstein constants agree: in dimension three each is constant on the connected manifold, by the contracted Bianchi identity, and the metrics agree in the nonempty exterior.

Both Einstein metrics are real analytic in harmonic coordinates by
\cite[Theorem~5.2]{deturck1981some}. For continuation of local isometries,
see \cite[Theorem~5.4]{deturck1981some} and
\cite[proof of Theorem~1.4]{graham2019x}.
Choose a smooth compact collar enlargement \(M^+\) of \(M\) strictly inside
\(\widetilde M\), with \(\p M^+\subset\Omega\). It is diffeomorphic to \(M\),
so \(\pi_1(M^+,\p M^+)=0\), and the two metrics agree on a full connected
collar of \(\p M^+\).

In this compact-boundary setting, the common-extension conclusion in
\cite[Section~3, following the proof of Theorem~3.1]{anderson2008unique},
read with \cite{anderson2010erratum}, realizes the two copies of \(M^+\)
as isometrically embedded domains \(\iota_j(M^+)\) in a common Einstein
manifold \(P\), extending the identity on a collar patch. Uniqueness of local
isometries extends this identification over the full connected collar.
Thus their embedded boundaries coincide, say
\(B=\iota_1(\p M^+)=\iota_2(\p M^+)\), and the domains have the same
interior side near \(B\). Each \(\iota_j(\operatorname{Int}M^+)\) is a
connected open and closed subset of \(P\setminus B\), where closedness
follows from compactness of \(M^+\). They intersect in the common interior
collar and hence coincide. Therefore \(\iota_2^{-1}\circ\iota_1\) is an
isometry of \((M^+,\bar g_1)\) onto \((M^+,\bar g_2)\), equal to the
identity in that collar.

By uniqueness of local isometries, this map is the identity throughout the
connected attached region \(M^+\setminus M\), where the metrics agree.
It maps the complementary copy of \(M\) to itself and fixes \(\p M\) by
continuity. Its restriction is the required \(\ol F\).

The Ricci-flat case is the same argument with Einstein constant zero.
\end{proof}

\appendix
\section{Derivation of the lapse and shift equations}\label{sect:appendix}

Throughout this appendix, let $(\ol{M}^{n+1},\og)$ be a smooth
Riemannian manifold and let a coordinate neighbourhood be foliated by
hypersurfaces
\[
    \Sigma_t=\{t=\text{constant}\}.
\]
Let $g_t$ be the metric induced on $\Sigma_t$, let $N$ be the chosen
$\og$-unit normal, and write
\[
    \dt=uN+\sigma,\qquad
    \sigma=\sigma^\alpha\p_\alpha,\qquad u>0.
\]
Here Greek indices refer to tangential coordinates $x^1,\ldots,x^n$ on the
leaves. Thus
\[
    N=u^{-1}(\dt-\sigma^\alpha\p_\alpha).
\]
The ambient metric has the form
\begin{equation}\label{coord-appendix:eq:ambient-metric}
    \og=u^2\,dt^2
    +g_{\alpha\beta}(dx^\alpha+\sigma^\alpha\,dt)
                         (dx^\beta+\sigma^\beta\,dt).
\end{equation}
Equivalently,
\begin{equation}\label{coord-appendix:eq:metric-components}
    \og_{\alpha\beta}=g_{\alpha\beta},\qquad
    \og_{t\alpha}=g_{\alpha\beta}\sigma^\beta,\qquad
    \og_{tt}=u^2+g_{\alpha\beta}\sigma^\alpha\sigma^\beta.
\end{equation}
The inverse metric is
\begin{equation}\label{coord-appendix:eq:inverse-metric}
    \og^{tt}=u^{-2},\qquad
    \og^{t\alpha}=-u^{-2}\sigma^\alpha,\qquad
    \og^{\alpha\beta}=g^{\alpha\beta}
                                  +u^{-2}\sigma^\alpha\sigma^\beta.
\end{equation}
We use the convention
\begin{equation}\label{coord-appendix:eq:second-fundamental-form}
    A(X,Y)=\og(\ol{\nabla}_X N,Y),\qquad
    X,Y\in T\Sigma_t,
\end{equation}
for the scalar second fundamental form of the leaves, and
\[
    H=\tr_{g_t}A
\]
for the mean curvature.

\begin{lem}\label{coord-appendix:lem:normal-acceleration}
For every vector field $X$ tangent to the leaves,
\begin{equation}\label{coord-appendix:eq:normal-acceleration}
    \og(\ol{\nabla}_N N,X)=-X(\log u).
\end{equation}
Equivalently,
\[
    \ol{\nabla}_N N=-\nabla^{g_t}\log u.
\]
\end{lem}

\begin{proof}
Since $N$ is unit, $\ol{\nabla}_N N$ is tangent to the leaves. Also
$N^\flat=u\,dt$, because $N$ is orthogonal to each $\Sigma_t$ and
$dt(N)=u^{-1}$. If $X$ is tangential, then $dt(X)=0$. Differentiating
$\og(N,X)=0$ along $N$ gives
\[
    \og(\ol{\nabla}_N N,X)
    =-\og(N,\ol{\nabla}_N X).
\]
Using torsion freeness,
\[
    \ol{\nabla}_N X=\ol{\nabla}_X N+[N,X].
\]
The first term on the right is orthogonal to $N$, and therefore
\[
    \og(\ol{\nabla}_N N,X)
    =-\og(N,[N,X])=-u\,dt([N,X]).
\]
Since $dt(X)=0$ and $dt(N)=u^{-1}$,
\[
    dt([N,X])=N(dt(X))-X(dt(N))=-X(u^{-1})=u^{-2}X(u).
\]
Thus
\[
    \og(\ol{\nabla}_N N,X)
    =-u^{-1}X(u)=-X(\log u).
\]
\end{proof}

\begin{lem}\label{coord-appendix:lem:metric-evolution}
The induced metrics $g_t$ satisfy
\begin{equation}\label{coord-appendix:eq:metric-evolution}
    \dt g_t=2uA+\mathcal{L}_\sigma g_t.
\end{equation}
In coordinates,
\begin{equation}\label{coord-appendix:eq:second-fundamental-coordinates}
    A_{\alpha\beta}=\frac{1}{2u}
    \bigl(\dt g_{\alpha\beta}
          -\nabla_\alpha\sigma_\beta-\nabla_\beta\sigma_\alpha\bigr),
\end{equation}
where $\sigma_\alpha=g_{\alpha\beta}\sigma^\beta$ and $\nabla$ is the
Levi-Civita connection of $g_t$.
\end{lem}

\begin{proof}
For tangential coordinate vector fields $\p_\alpha,\p_\beta$,
the coordinate brackets with $\dt$ vanish. Hence
\[
    \dt g_{\alpha\beta}
    =(\mathcal{L}_{\dt}\og)(\p_\alpha,\p_\beta).
\]
Since $\dt=uN+\sigma$, the normal part gives
\[
    (\mathcal{L}_{uN}\og)(\p_\alpha,\p_\beta)
    =2uA_{\alpha\beta},
\]
and the tangential part gives $(\mathcal{L}_\sigma g_t)_{\alpha\beta}$.
This proves \eqref{coord-appendix:eq:metric-evolution}.
\eqref{coord-appendix:eq:second-fundamental-coordinates} follows by writing the Lie
derivative of $g_t$ covariantly.
\end{proof}

\begin{lem}\label{coord-appendix:lem:laplacian-variation}
Let $g_s$ be a smooth family of metrics on a fixed manifold, and write
\[
    h=\left.\frac{d}{ds}\right|_{s=0}g_s.
\]
For a fixed smooth function $f$,
\begin{equation}\label{coord-appendix:eq:laplacian-variation}
    \left.\frac{d}{ds}\right|_{s=0}\D_{g_s}f
    = D\D_g[h]f:=-\langle h,\nabla^2 f\rangle_g
     +\left\langle-\divergence_g h+\frac12 d(\tr_g h),df\right\rangle_g.
\end{equation}
Here
\[
    (\divergence_g h)_j=\nabla^i h_{ij}.
\]
\end{lem}

\begin{proof}
In local coordinates,
\[
    \D_g f=g^{ij}(\p_i\p_j f-\Gamma^k_{ij}\p_k f).
\]
The inverse metric variation is $\dot{g}^{ij}=-h^{ij}$, and
\[
    \dot{\Gamma}^k_{ij}=\frac12 g^{k\ell}
    (\nabla_i h_{j\ell}+\nabla_j h_{i\ell}-\nabla_\ell h_{ij}).
\]
Substituting these two identities into the coordinate expression for
$\D_g f$ gives \eqref{coord-appendix:eq:laplacian-variation}.
\end{proof}

\begin{lem}\label{coord-appendix:lem:lie-laplacian}
If $V$ is a vector field tangent to a leaf and $f$ is a smooth function on
that leaf, then
\begin{equation}\label{coord-appendix:eq:lie-laplacian}
    D\D_g[\mathcal{L}_V g]f=V(\D_g f)-\D_g(Vf).
\end{equation}
In particular,
if $\D_g f=0$, then
\begin{equation}\label{coord-appendix:eq:lie-laplacian-harmonic}
    D\D_g[\mathcal{L}_V g]f=-\D_g(Vf).
\end{equation}
\end{lem}

\begin{proof}
Let $\Phi_s$ be the flow of $V$. Since $\Phi_s^*g$ has first variation
$\mathcal{L}_V g$,
\[
    \D_{\Phi_s^*g}f
    =\Phi_s^*\bigl(\D_g((\Phi_s^{-1})^*f)\bigr).
\]
Differentiating at $s=0$ gives \eqref{coord-appendix:eq:lie-laplacian}.
\end{proof}
The next proposition is a restatement of Proposition \ref{prop: elliptic equation for lapse and shift}, and proves it.
\begin{prop}\label{coord-appendix:prop:lapse-shift}
Assume that $(t,x^1,\ldots,x^n)$ are $H$-harmonic coordinates and that each
leaf $\Sigma_t$ is minimal. Thus
\[
    H=0,\qquad \D_{g_t}x^\alpha=0,\qquad \alpha=1,\ldots,n.
\]
Then the lapse satisfies
\begin{equation}\label{coord-appendix:eq:lapse}
    \D_{g_t}u+
    \bigl(|A|_{g_t}^2+\Ric_{\og}(N,N)\bigr)u=0.
\end{equation}
The shift satisfies, for each $\alpha=1,\ldots,n$,
\begin{equation}\label{coord-appendix:eq:shift}
\begin{split}
    \D_{g_t}\sigma^\alpha
    &=-2u\langle A,\nabla^2x^\alpha\rangle_{g_t}
      -2A(\nabla u,\nabla x^\alpha)\\
    &\qquad-2u(\divergence_{g_t}A)(\nabla x^\alpha).
\end{split}
\end{equation}
Here $\nabla$, $\D_{g_t}$, and $\divergence_{g_t}$ are taken on the leaf
$(\Sigma_t,g_t)$, and
\[
    (\divergence_{g_t}A)(X)=\nabla^\beta A_{\beta\gamma}X^\gamma.
\]
\end{prop}

\begin{proof}
The lapse equation is the normal variation formula for the mean curvature.
With the convention \eqref{coord-appendix:eq:second-fundamental-form}, variation in the
normal direction $fN$ gives
\[
    \left.\frac{d}{ds}\right|_{s=0}H_s
    =-\D_{g_t}f
     -\bigl(|A|_{g_t}^2+\Ric_{\og}(N,N)\bigr)f.
\]
The tangential part of the deformation contributes only $\sigma(H)$.
Since $H\equiv0$ for the foliation, the deformation vector field
$\dt=uN+\sigma$ gives
\[
    0=-\D_{g_t}u
      -\bigl(|A|_{g_t}^2+\Ric_{\og}(N,N)\bigr)u,
\]
which is \eqref{coord-appendix:eq:lapse}.

We next prove the shift equation. Each coordinate function
$x^\alpha$ is harmonic on each leaf, and $x^\alpha$ does not depend on $t$ as a function of the coordinates $(t,x)$. Therefore
\begin{eqnarray}\label{eq:0 is equal to stuff}
    0=\dt(\D_{g_t}x^\alpha) = D\Delta_g[\partial_t g_t]x^\alpha = D\Delta_g[2uA]x^\alpha+ D\Delta_g[\mathcal L_\sigma g_t ]x^\alpha .
\end{eqnarray}
The second equality is Lemma \ref{coord-appendix:lem:laplacian-variation}, the third equality is due to Lemma \ref{coord-appendix:lem:metric-evolution}.

Using Lemma~\ref{coord-appendix:lem:lie-laplacian} and $\D_{g_t}x^\alpha=0$, the
contribution of $\mathcal{L}_\sigma g_t$ is
\[
   D\Delta_g[\mathcal L_\sigma g_t ]x^\alpha =  -\D_{g_t}(\sigma(x^\alpha))=-\D_{g_t}\sigma^\alpha.
\]
Therefore, \eqref{eq:0 is equal to stuff} becomes
\begin{equation}\label{coord-appendix:eq:shift-variation}
    \D_{g_t}\sigma^\alpha=D\D_{g_t}[2uA]x^\alpha.
\end{equation}
Applying Lemma~\ref{coord-appendix:lem:laplacian-variation} with $h=2uA$ gives
\begin{align*}
    D\D_{g_t}[2uA]x^\alpha
    &=-2u\langle A,\nabla^2x^\alpha\rangle_{g_t}
      +\left\langle-\divergence_{g_t}(2uA)
        +\frac12 d(2uH),dx^\alpha\right\rangle_{g_t}\\
    &=-2u\langle A,\nabla^2x^\alpha\rangle_{g_t}
      -2A(\nabla u,\nabla x^\alpha)
      -2u(\divergence_{g_t}A)(\nabla x^\alpha),
\end{align*}
because $H\equiv0$. Combining this identity with
\eqref{coord-appendix:eq:shift-variation} proves \eqref{coord-appendix:eq:shift}.
\end{proof}

\begin{rmk}\label{coord-appendix:rem:coordinate-shift}
In the harmonic coordinates $x^1,\ldots,x^n$ on a leaf,
\[
    \nabla_\mu\nabla_\nu x^\alpha=-\Gamma^\alpha_{\mu\nu}.
\]
Therefore the shift equation \eqref{coord-appendix:eq:shift} can be written as
\begin{equation}\label{coord-appendix:eq:coordinate-shift}
    \D_{g_t}\sigma^\alpha
    =2uA^{\mu\nu}\Gamma^\alpha_{\mu\nu}
     -2A^{\mu\alpha}\nabla_\mu u-2u\nabla_\mu A^{\mu\alpha},
\end{equation}
where indices are raised using $g_t$.
\end{rmk}

\bibliographystyle{alpha}
\bibliography{ref}

\end{document}